\documentclass[twoside,a4paper,reqno,11pt]{amsart} 
\usepackage[top=28mm,right=30mm,bottom=28mm,left=30mm]{geometry}

\usepackage{amsfonts, amsmath, amssymb, mathrsfs, bm, latexsym, stmaryrd, array, mathtools, bold-extra}

\usepackage{color}
\usepackage{xcolor}

\usepackage[pagebackref]{hyperref}

\usepackage{bbm}
\usepackage[all,cmtip]{xy}

\usepackage{enumerate}
\usepackage{comment}
\usepackage[shortlabels]{enumitem}

\newtheorem{theorem}{Theorem} 
\newtheorem*{conj*}{Conjecture}

\newtheorem{hypothesis}[theorem]{Hypothesis}

\newtheorem{thm}{Theorem}[section] 
\newtheorem{prop}[thm]{Proposition} 
\newtheorem{lem}[thm]{Lemma}
\newtheorem{cor}[thm]{Corollary} 

\newtheorem{hyp}[thm]{Hypothesis}

\theoremstyle{definition}
\newtheorem{rem}[thm]{Remark}

\newtheorem{ex}[thm]{Example}

\newtheorem{problem}[theorem]{Problem}

\newtheorem{conjecture}[theorem]{Conjecture}

\theoremstyle{definition}

\theoremstyle{remark}

\DeclareMathOperator{\Irr}{Irr}

\DeclareMathOperator{\GL}{GL}

\DeclareMathOperator{\kernel}{Ker}

\newcommand{\Lin}{{\mathrm{ Lin}}}

\newcommand{\la}{\lambda}

\DeclareMathOperator{\Sym}{S}

\newcommand{\HH}{\mathcal{H}}

\newcommand{\Centralizer}{\mathrm{C}}
\newcommand{\Center}{\mathrm{Z}}

\def\norm#1#2{{\bf N}_{#1}(#2)}

\numberwithin{equation}{section}

\newcommand{\Alt}{{\mathrm {A}}}

\def\irr#1{{\rm Irr}(#1)}
\def\C{{\Bbb C}}

\begin{document}

\title[A conjugacy class that is a difference]{Groups with a conjugacy class that is the difference of two normal subgroups}

\author[M. L. Lewis]{Mark L. Lewis}
\address{Department of Mathematical Sciences, Kent State University, Kent, OH 44242, USA}
\email{lewis@math.kent.edu}

\author[L. Morotti]{Lucia Morotti}
\address{Department of Mathematics, University of York, York, YO10 5DD, UK}
\email{lucia.morotti@york.ac.uk} 

\author [E. Pacifici]{Emanuele Pacifici}
\address{Dipartimento di Matematica e Informatica U. Dini, Universit\`a degli Studi di Firenze, Viale Morgagni 67/A, 50134 Firenze, Italy.}
\email{emanuele.pacifici@unifi.it}

\author [L. Sanus]{Lucia Sanus}
\address{Departament de Matem\`atiques, 
Universitat de Val\`encia,
46100 Burjassot, Val\`encia, Spain. }
\email{lucia.sanus@uv.es}

\author[H. P. Tong-Viet]{Hung P. Tong-Viet}
\address{Department of Mathematics and Statistics, Binghamton University, Binghamton, NY 13902-6000, USA}
\email{htongvie@binghamton.edu}

\renewcommand{\shortauthors}{M. L. Lewis et al.}
\begin{abstract}  
We consider finite groups having a conjugacy class that is the difference of two normal subgroups.  That is, suppose $G$ is a group and $M$ and $N$ are normal subgroups so that $N < M$, and suppose that there is an element $g \in G$ so that the conjugacy class of $g$ is $M \setminus N$.  We find a character-theoretic characterization of this condition, and we determine some structural properties of groups with such a conjugacy class.  If we add the condition that $M/N$ is the unique minimal normal subgroup of $G/N$, then we obtain a generalization of a result by S.M. Gagola.
\end{abstract}

\thanks{While starting to work on this paper the second author was working at the Mathematisches Institut of the Heinrich-Heine-Universit\"{a}t D\"usseldorf. The second author was supported by the Royal Society grant URF$\backslash$R$\backslash$221047. The third author is partially supported by InDaM-GNSAGA and by the italian project PRIN 2022-2022PSTWLB - Group Theory and Applications - CUP B53D23009410006. The research of the fourth author is supported by Ministerio de Ciencia e Innovación (PID2022-137612NB-I00 funded by MCIN/AEI/10.13039/501100011033 and `ERDF A way of making Europe') and Generalitat Valenciana CIAICO/2021/163. }





\maketitle

\centerline{\it In honor of Alan Camina on his 85th birthday}

\section{Introduction}

Throughout this paper, all groups are finite groups unless otherwise stated.  Let $G$ be a group and let $C$ be a conjugacy class of $G$; then $M = \langle C \rangle$ is a normal subgroup of $G$ and, for any normal subgroup $N$ with $N < M$, we see that $C \subseteq M \setminus N$.  

In this paper, we consider the conditions under which we have equality between $C$ and $M \setminus N$.  When $C = M \setminus N$, we will prove in Lemma \ref{lem:M le G} that $M/N$ is a chief factor for $G$ and that in fact $N$ is the unique normal subgroup of $G$ so that $M/N$ is a chief factor of $G$.  Our goal in this paper is to find an equivalent condition on the irreducible characters of $G$.

Given a group $G$ and an element $g \in G$, we denote by $\irr G$ the set of the irreducible complex characters of $G$, and by $g^G$ the conjugacy class of $g$ in $G$.  When $N$ is a normal subgroup of $G$, we write $\irr {G \mid N}$ for the set of irreducible characters of $G$ whose kernels do not contain $N$.  We obtain the following character-theoretic characterization of the aforementioned situation.

\begin{theorem}\label{intro-standard conj}
Let $G$ be a group, let $N < M$ be normal subgroups of $G$, and let $g \in M \setminus N $.  Then $g^G = M \setminus N$ if and only if every character in $\irr {G/N}$ is constant on $M/N \setminus \{ N \}$ and every character in $\irr {G \mid N}$ vanishes on $g$.
\end{theorem}

We will see that, in the situation of the previous theorem, $M/N$ is an elementary abelian $p$-group for a prime $p$, and that $g$ has $p$-power order.  We will also show that, while there are nonsolvable groups $G$ that satisfy this condition, in fact $M$ must be solvable.

\begin{theorem}\label{intro-standard res}
Let $G$ be a group, let $N < M$ be normal subgroups of $G$, and let $g \in M \setminus N$ be such that $g^G = M \setminus N$.   Then there is a prime $p$ so that the following are true:
\begin{enumerate}[{\rm (a)}] 
\item  $g$ has $p$-power order, so every element in $M \setminus N$ has $p$-power order.
\item  $M$ is solvable and has a normal $p$-complement.
\end{enumerate}
\end{theorem}

Next, we show that $2$-groups are the only $p$-groups for which this situation can occur.

\begin{theorem} \label{intro-conj p groups}
Let $G$ be a $p$-group for some prime $p$, and let $N < M$ be normal subgroups of $G$ with $g \in M \setminus N$ so that $g^G = M \setminus N$.  Then $p = 2$ and $|M:N| = 2$.
\end{theorem}

We now return to the situation of a general group.  

\begin{theorem} \label{intro-mutual}
Let $G$ be a group, and let $N$ be a normal subgroup of $G$.  Then the following are equivalent:
\begin{enumerate}[$(1)$]
\item There is a normal subgroup $M$ of $G$ with $N < M$ and an element $g \in M \setminus N$ such that $g^G = M \setminus N$ and $|\irr {G/N \mid M/N}|= 1$. 
\item There is an element $g \in G \setminus N$, and $\chi \in \irr G$ with $N \le \ker \chi$, such that $\chi$ vanishes on all of the $G$-conjugacy classes in $G \setminus N$ except $g^G$ and every character in $\irr {G \mid N}$ vanishes on $g$.
\item There is normal subgroup $M$ of $G$ with $N < M$, and a unique $\chi \in \irr G$ that does not vanish on $M \setminus N$ and does not have the elements of $M \setminus N$ in its kernel.
\end{enumerate}
\end{theorem}

As shown by Proposition~\ref{degenerate}, in the case when $G/N$ has order $2$ (and $M=G$) the three equivalent conditions of the above theorem are also equivalent to the fact that $G$ is either of order $2$ or a Frobenius group with kernel $N$ and complements of order $2$.

We also are able to characterize some properties of the group and character for the situation considered in Theorem \ref{intro-mutual}.  With this in mind, we make the following hypothesis:

\begin{hypothesis}\label{Hypothesis 1}  
Let $G$ be a group and let $N < M$ be normal subgroups of $G$ with $g \in M \setminus N$.  Assume that  $g^G = M \setminus N$ and $\irr {G/N \mid M/N} = \{ \chi\}$.
\end{hypothesis}

We now obtain some arithmetical results about groups that satisfy Hypothesis \ref{Hypothesis 1}.  Let $p$ be a prime.  When $n$ is an integer, we write $n_p$ for the largest power of $p$ that divides $n$ and we set $n_{p'}=n/n_p$.

\begin{theorem}\label{intro-mutual 1}
Assume Hypothesis $\ref{Hypothesis 1}$.  Then there is a prime $p$ so that the following are true:
\begin{enumerate}[{\rm (a)}]
\item  $\displaystyle |\Centralizer_G (g)| = \frac {|G:M| (|G:M| + \chi (1)^2)}{\chi (1)^2} = |\Centralizer_{G/N} (gN)| = |G/N|_p$.  In particular, $\Centralizer_G (g)$ is a $p$-group.
\item  $\displaystyle \chi (x) = \left\{ \begin{array}{ll}|G:M|_{p'} \sqrt {|G:M|_p} & {\rm if~~} x \in N \cr
	\displaystyle -\frac {|G:M|}{\chi (1)} = -\sqrt{|G:M|_p} & {\rm if~~} x \in g^G = M \setminus N \cr 
	0 & {\rm if~~} x \in G \setminus M .\end{array} \right. $ 
\end{enumerate}
\end{theorem}

Our setting is a generalization of the one that S.M. Gagola studies in \cite{Gagola}. In that paper, the author considers groups having a nonlinear irreducible character $\chi$ that vanishes on all but two conjugacy classes of $G$.  He shows that $G$ has a unique minimal normal subgroup $M$, this $M$ is an elementary abelian $p$-group for a prime $p$, and the nonidentity elements in $M$ form a single conjugacy class of $G$.  He also shows that $\chi$ is the only irreducible character of $G$ that does not have $M$ in its kernel.  Furthermore, he shows that $G$ has a conjugacy class which consists of all nonidentity elements of a minimal normal subgroup if and only if $G$ has an irreducible character that vanishes on all but two conjugacy classes.  In other words, Gagola handles the case for Theorems \ref{intro-mutual} and \ref{intro-mutual 1} when $M$ is a minimal normal subgroup and $N = 1$.  In particular, we say that a group $G$ has a {\it Gagola character} if it has a nonlinear irreducible character $\chi$ that vanishes on all but two conjugacy classes of $G$, and we say $G$ is a {\it Gagola group} if it has a Gagola character. Note that our definition of a Gagola character differs from Gagola's in \cite{Gagola}, in
that we are requiring $\chi$ to be nonlinear.  This has the highly
beneficial effect of ruling out the very degenerate case when $|G| = 2$.

There are two other situations that we wish to highlight in this paper.  The first is to generalize the condition characterizing Gagola groups, by assuming that  a group $G$ has an irreducible character $\chi$ which vanishes on all but one {\it noncentral} conjugacy class of $G$.   We obtain the following generalization of Gagola's theorem, which, as can be read from the last claim of the statement, fits in our analysis about conjugacy classes that are the difference of two normal subgroups. 

\begin{theorem} \label{th:character}
Let $G$ be a group, and suppose $\chi$ is a nonlinear irreducible character of $G$.  Then the following properties are equivalent:
\begin{enumerate}[$(1)$]
\item There exists $g \in G \setminus \Center (G)$ such that $\chi$ vanishes on all noncentral conjugacy classes of $G$ except $g^G$.
\item We have $\Center (G) \leq \kernel(\chi)$, and there exists a noncentral element $g$ of $G$ such that, for every character $\psi \in \Irr(G)$ with $\psi \ne \chi$, either $g \in \kernel(\psi)$ when $\psi \in \Irr(G/\Center (G))$ or otherwise $\psi (g) = 0$.
\end{enumerate} 
Furthermore, if the above properties hold, then we have $\Center (G) = \kernel (\chi) < G'$, the value $\chi (g)$ is $-|\Center (G)| \chi(1)/|g^G| \in \mathbb{Z}$, and $ |\Centralizer_G (g)| = |\Centralizer_{G/\Center (G)} (g\Center (G))|$ is a $p$-power for a suitable prime $p$. In particular, $\Center(G)$ is a $p$-group. Finally, $G/\Center(G)$ has a unique minimal normal subgroup $M/\Center(G)$, and $M\setminus\Center(G)$ is a conjugacy class of $G$.
\end{theorem}

Note that the hypothesis of Theorem \ref{th:character} is one step from groups that have an irreducible character that vanishes on every noncentral element.  A group with such a character is called a {\it group of  central type}.  In \cite{DeJa}, F.R. DeMeyer and G. Janusz prove that the Sylow subgroups of groups of central type are also of central type, and R.B. Howlett and I.M. Isaacs in \cite{IsHo} build on the work of Gagola from \cite{Gafr} to show that every group of central type is solvable.  We note that groups with a Gagola character need not be solvable.  A nonsolvable example is constructed in \cite{nonsolv}.  There are also nonsolvable examples described in Theorem 1.20 of \cite{GGLMNT}.  In Section \ref{examples}, we will present a nonsolvable example for Theorem~\ref{th:character}.

It is known that there seems to be a dual theory between conjugacy classes and irreducible characters of a group.   
%
Recall that an element $g\in G$ is called a {\it Camina element} (or an {\it anticentral element}) (see \cite[Lemma~2.1]{Lewis} or \cite[Proposition~1.1]{Ladisch}, respectively) if every nonlinear irreducible character of $G$ vanishes at $g$. It is known that groups with such an element are solvable (see \cite[Theorem~4.3]{Ladisch}). 
The next step is to consider an element of $G$ on which all but one of the nonlinear irreducible characters vanish.  In this situation, 
%
%
%
we are able to show that $G$ must be solvable.  Notice that the alternating group $A_5$ has a conjugacy class that is a zero of all but two nonlinear irreducible characters, so we cannot extend this condition any further and still obtain solvability.  We now specialize Theorems \ref{intro-mutual} and \ref{intro-mutual 1} to the case where $M = G'$, and we obtain the following theorem:

\begin{theorem}\label{th:class}
Let $G$ be a group. Assume there exist an element $g \in G \setminus \Center (G)$ and a nonlinear character $\chi \in \irr G$ such that $\chi(g)\neq 0$ and, for every nonlinear irreducible character $\psi \neq \chi$ of $G$, we have $\psi (g) = 0$. Then the following conclusions hold.
\begin{enumerate}[{\rm (a)}]
\item $g\in G'$ and $\chi (g) = -|G:G'|/\chi(1)$.
\item $G' = \kernel(\chi) \cup g^G$ and $\chi$ vanishes off of $G'$.
\item $\chi$ is the unique nonlinear irreducible character of $G/\kernel(\chi)$. 
\item $G$ is solvable.
\end{enumerate}
\end{theorem}

In Section \ref{secn:background}, we provide some background information.  We then review the results of Gagola regarding groups with an irreducible character that vanishes on all but two conjugacy classes and results about Camina pairs.  We continue in Section \ref{secn:between} where we prove Theorems \ref{intro-standard conj}, \ref{intro-standard res},  \ref{intro-conj p groups}, \ref{intro-mutual}, and \ref{intro-mutual 1}.  We start Section \ref{occurs 1} by proving Theorems \ref{th:character} and \ref{th:class}.  We then consider a number of special cases of those theorems.  In Section \ref{occurs 2}, we discuss $p$-groups that have a Gagola character and their extensions.  Finally, in Section \ref{examples}, we present examples on non $p$-groups having a Gagola charcter (including a nonsolvable example) and their extensions.


\section{Background} \label{secn:background}

A conjugacy class $C$ that coincides with a set of the form $M \setminus N$ where $N < M$ are normal subgroups of $G$ is, in some sense, a conjugacy class of maximal size.  By way of explanation, let $C$ be a conjugacy class of $G$, and take $M = \langle C \rangle$ to be the (normal) subgroup of $G$ generated by $C$.  We show that $C$ is contained in the difference $M\setminus N$ for every normal subgroup $N$ of $G$ that is properly contained in $M$.

\begin{lem} \label{conj containment}
Let $G$ be a group, let $C$ be a conjugacy class of $G$, and let $M = \langle C \rangle$.  If $N$ is a normal subgroup of $G$ such that $N < M$, then $C \subseteq M \setminus N$. 
\end{lem}

\begin{proof}
Since $M = \langle C \rangle$, we have that $C$ is not contained in $N$.  Since $N$ is a union of conjugacy classes of $G$, we have that $C \cap N$ is empty, and hence, $C \subseteq M \setminus N$.
\end{proof}

We now consider the situation when $g^G = M \setminus N$ for some $g \in M \setminus N$.

\begin{lem}\label{lem:M le G}
Let $G$ be a group and let $C$ be a conjugacy class of $G$.  Suppose that $N < M$ are normal subgroups of $G$ such that $C = M \setminus N$.  Then the following are true:
\begin{enumerate}[{\rm (a)}]
\item If $M \not\le G'$, then $N = M \cap G'$, $|M:N| = 2$, and $C = aN$ for any $a \in C$.
\item $M = \langle C \rangle$.
\item $M/N$ is a chief factor of $G$.
\item If $K$ is normal in $G$ so that $K < M$, then $K \le N$.  In particular, $N$ is the unique normal subgroup of $G$ contained in $M$ so that $M/N$ is a chief factor for $G$.
\end{enumerate}
\end{lem}

\begin{proof}
Since we are assuming $C = M \setminus N$, we have $M = \langle C \rangle \cup N$; since $\langle C \rangle$ is not contained in $N$, the converse must be true, hence $M = \langle C \rangle$ and (b) is proved. 

Next, let $K$ be normal in $G$ so that $K < M$.  By Lemma \ref{conj containment}, we have $C \subseteq M \setminus K$.  Since $C = M \setminus N$, this implies $M \setminus N \subseteq M \setminus K$ and hence $K \subseteq N$.  It follows that $N$ contains every normal subgroup of $G$ that is proper in $M$.  This implies that $M/N$ is a chief factor for $G$, and (c), (d) are proved.

Finally, for (a), let us assume that $M$ is not contained in $G'$. For $a\in C$, it is easily seen that $M\setminus N=C\subseteq aG'$, and $M=\langle C\rangle=\langle a^G\rangle\subseteq \langle a\rangle G'$. In particular, $a\not\in G'$; moreover, since $a^{-1}\not\in N$, we have $a\in C$ and therefore $a^{-1}\in aG'$. Thus we get $aG'=a^{-1}G'$, and so $a^2\in G'$. It follows that $|\langle a\rangle G': G'|=2$, and hence $|M:M\cap G'|=2$. Now, by (c), we must have $N=M\cap G'$ and the conclusion follows.
\end{proof}

To see examples of groups as in Lemma~\ref{lem:M le G} such that $M$ is not contained in $G'$, take $M=G$ to be a Frobenius group whose Frobenius complements have order $2$. In this case, $N$ will be the Frobenius kernel $G'$, and if $g \in G \setminus N$, then $g^G = G \setminus N$. As another example of this kind, but with $M\neq G$, one can consider $G$ to be the dihedral group of order $8$, $M$ a maximal subgroup of $G$ and $N=G'$.  We note that the situation when $M$ is not contained in $G'$, which yields $g^G = gN$, is handled in Lemma 3 and Corollary 4 of \cite{flat}, and discussed further in Proposition~\ref{degenerate}.

Still assuming that $C$ is a conjugacy class of $G$ and $M=\langle C\rangle$, let $N_1, \ldots, N_k$ be the various normal subgroups of $G$ contained in $M$ so that each $M/N_i$ is a chief factor for $G$.  
Observe that $C \subseteq M \setminus N_i$ for each $i$.  In Theorem \ref{intro-standard conj}, we are finding equivalent conditions on $C$ and on the irreducible characters of $G$ for $k$ to equal $1$ and equality to hold.  It would be interesting to see if one can find conditions on either a conjugacy class $C$ of $G$ or the irreducible characters of $G$ that yield either that $k = 1$ (without requiring the class equal $M \setminus N_1$) or that equality holds so that $\displaystyle C =  G \setminus \cup_{i=1}^k N_i$.


When $G$ is a Gagola group, we have the following (Lemma~2.1 in \cite{Gagola} or Theorem~4.16 in \cite{Lewis}): $G$ has a unique minimal normal subgroup $N$ and $N$ is an elementary abelian $p$-group for some prime $p$, the zeros of $\chi$ are precisely the elements of $G \setminus N$, and $N \setminus \{ 1 \}$ is a conjugacy class of $G$ (which implies, by Corollary 6.33 of \cite{Isaacs}, that the conjugation action of $G$ on $\irr N \setminus \{ 1_N \}$ is transitive).  We summarize this result  in the following theorem:

\begin{thm}\label{gagola} \emph{[Lemma 2.1 of \cite{Gagola}]}
Let $G$ be a group with $|G| > 2$.  Assume that $G$ has an irreducible character $\chi$ that does not vanish on exactly two conjugacy classes of $G$.  Then the following hold:
\begin{enumerate}[{\rm (a)}]
\item $\chi$ is the unique irreducible character that vanishes on all but two conjugacy classes of $G$.
\item $\chi$ is the unique irreducible faithful character of $G$.
\item $G$ has a unique minimal normal subgroup $N$.
\item $N$ is an elementary abelian $p$-group for some prime $p$.
\item $\chi$ vanishes on $G \setminus N$ and does not vanish on $N$.
\item The action of $G$ by conjugation on $N$ is transitive on the nonidentity elements.
\end{enumerate}
\end{thm}

To introduce more properties of groups with Gagola characters, we also need to recall the concept of a \emph{Camina pair}, which is a generalization of the concept of a Camina element.  Let $N \neq 1$ be a proper normal subgroup of $G$: we say that $(G,N)$ is a Camina pair if, for every element $g \in G \setminus N$, we have $gN\subseteq g^G$. There are a number of equivalent conditions for being a Camina pair. The following lemma is stated as Lemma 4.1 in \cite{Lewis}. 

\begin{lem}\label{Camina pair-equiv}
Let $N\neq 1$ be a proper normal subgroup of a group $G$.  Then the following are equivalent:
\begin{enumerate}[$(1)$]
\item $(G,N)$ is a Camina pair.
\item If $x \in G \setminus N$, then $|\Centralizer_G (x)| = |\Centralizer_{G/N} (xN)|$.
\item If $xN$ and $yN$ are conjugate in $G/N$ and nontrivial, then $x$ is conjugate to $y$ in $G$.
\item For all $x \in G \setminus N$ and $z \in N$, there exists an element $y \in G$ such that $[x,y] = z$.
\item Every $\chi \in \irr {G \mid N}$ vanishes on $G \setminus N$. 
\end{enumerate}
\end{lem}

Now we can see that Gagola groups and Camina pairs are closely related. The following lemma is known, but we believe that it has not been explicitly proved in the literature; as it will be helpful to have it, we provide an explicit proof.   In particular, this shows that groups with a Gagola character are precisely the Camina pairs where there is a unique irreducible character lying in $\irr {G \mid N}$.  Note that these Camina pairs have $N$ that is minimal normal, so in particular, $N$ is an elementary abelian $p$-group for some prime $p$.

\begin{lem}\label{gagola pairs}
Let $N\neq 1$ be a proper normal subgroup of a group $G$ and let $\chi \in \irr G$.  Then the following are equivalent:
\begin{enumerate}[$(1)$]
\item $\chi$ is a Gagola character for $G$  and $N$ is the unique minimal normal subgroup of~$G$.
\item $(G,N)$ is a Camina pair and $\irr {G \mid N} = \{ \chi \}$.
\end{enumerate}	
\end{lem}

\begin{proof}
Suppose first that $G$ has a Gagola character $\chi$ and $N$ is the unique minimal normal subgroup of $G$.  Let $\theta \in \irr N$ be an irreducible constituent of $\chi_N$.  We 
know that $\theta^G$ vanishes on $G \setminus N$.  
Moreover, setting $m = [\chi_{N},\theta]$, taking $T$ to be the stabilizer of $\theta$ in $G$, writing $t = |G:T|$, and denoting by $\theta_1, \dots, \theta_t$ the distinct $G$-conjugates of $\theta$, we have $\chi_N = m \sum_{i=1}^{t} \theta_i$.  Since we also have $(\theta^G)_N = |T:N| \sum_{i=1}^t \theta_i$, we see that $(\theta^G)_N$ is a complex multiple of $\chi_N$.  

Since $\chi$ is a Gagola character, we know that it also vanishes on $G \setminus N$. This implies that $\theta^G$ is a complex multiple of $\chi$. In other words, $\chi$ is the unique irreducible constituent of $\theta^G$.  By Lemma 2.1 of \cite{Gagola}, we know that $\theta$ is conjugate to all of the nontrivial characters of $N$, and so, $\irr {G \mid N} = \{ \chi \}$.  Since $\chi$ vanishes on $G \setminus N$, we see that every irreducible character in $\irr {G \mid N}$ vanishes on $G \setminus N$, and so by Lemma \ref{Camina pair-equiv}, we have that $(G,N)$ is a Camina pair.

Conversely, suppose that $(G,N)$ is a Camina pair and $\irr {G \mid N} = \{ \chi \}$.  Thus, the action of $G$ on $\irr N$ has two orbits.  Using Corollary 6.33 of \cite{Isaacs}, we see that the action of $G$ on the conjugacy classes of $N$ has two orbits.  This implies that all of the nonidentity elements of $N$ are conjugate.  It follows that $\chi$ vanishes on all but two conjugacy classes $\{1\}$ and $N \setminus \{ 1\}$, which implies that $G$ has a Gagola character and $N$ is the unique minimal normal subgroup of $G$.
\end{proof}


\section {Conjugacy classes that are the differences of two normal subgroups} \label{secn:between}

In this section, we prove the results regarding a conjugacy class which is the difference of two normal subgroups.  We begin with the special case where the smaller normal subgroup is the trivial subgroup, and hence, the larger normal subgroup is a minimal normal subgroup.   This is the more general version of the situation studied by Gagola.  In Gagola's scenario, the minimal normal subgroup is the unique minimal normal subgroup of $G$.  We are not assuming that the minimal normal subgroup is unique.  Thus, it is not surprising that we obtain a weaker conclusion. 

\begin{lem}\label{min standard conj}
Let $G$ be a group and let $M$ be a normal subgroup of $G$.  Suppose $g \in M \setminus \{ 1 \}$.  Then the following are true:
\begin{enumerate}[{\rm (a)}]
\item If $g^G = M \setminus \{ 1 \}$, then $M$ is a minimal normal subgroup, and is an elementary abelian $p$-group for some prime $p$.  Also, $G/\Centralizer_G (M)$ acts transitively on $M \setminus \{ 1 \}$.
\item $g^G = M \setminus \{ 1 \}$ if and only if every character in $\irr {G}$ is constant on $M \setminus \{ 1 \}$.
\end{enumerate}   
\end{lem}

\begin{proof}
Suppose that $g^G = M \setminus \{ 1 \}$.  By Cauchy's theorem, we know that some element of $M$ has order $p$ for some prime $p$, and since all the elements in $M \setminus \{ 1 \}$ are conjugate, all nonidentity elements of $M$ have order $p$, and so, $M$ is an elementary abelian $p$-group.  We see that $G$ acts transitively on $M \setminus \{ 1 \}$.   It follows that $G/\Centralizer_G (M)$ acts transitively on $M \setminus \{ 1 \}$, and $M$ is a minimal normal subgroup of $G$.   

Thus, we see that if $g^G = M \setminus \{ 1 \}$, then $G$ acts transitively on the nonprincipal characters in $\irr M$.   Hence, if $\chi \in \irr {G \mid M}$, then the irreducible constituents of $\chi_M$ consist of all of the nonprincipal irreducible characters of $M$.  We use the fact that characters are class functions.  Because $\chi$ is a class function, so it follows that $\chi_M$ is a multiple of $\rho_M - 1_M$ where $\rho_M$ is the regular character of $M$, and thus, $\chi_M$ is constant on $M \setminus \{ 1 \}$.  This shows that all of the characters in $\irr {G \mid M}$ are constant on $M \setminus \{1 \}$.  On the other hand, all of the characters in $\irr {G/M}$ are constant on $M$.  Therefore, we can use the fact that $\irr G = \irr {G/M} \cup \irr {G \mid M}$ to see that if all characters in $\irr {G \mid M}$ are constant on $M \setminus \{ 1 \}$, then so are all characters of $G$.  This proves the forward direction of (b).

Conversely, suppose that every character in $\irr {G}$ is constant on $M \setminus \{ 1 \}$.  Now, because the irreducible characters form a basis for class functions of $G$, we deduce that all of the nonidentity elements of $M$ are $G$-conjugate, and so, $g^G = M \setminus \{ 1 \}$.
\end{proof}

We now consider the general situation of a conjugacy class that is the difference of two normal subgroups.  Observe that this theorem includes Theorem \ref{intro-standard conj} from the Introduction.

\begin{thm}\label{standard conj}
Let $G$ be a group and let $N < M$ be normal subgroups of $G$ with $g \in M \setminus N$.  Then the following are equivalent:
\begin{enumerate}[$(1)$]
\item $g^G = M \setminus N$. 
\item  Every character in $\irr {G/N}$ is constant on $M/N \setminus \{ N \}$ and every character in $\irr {G \mid N}$ vanishes on $g$.
\item  $|G:\Centralizer_G (g)| = |M| - |N|$.
\item  $gN$ is conjugate in $G/N$ to every element in $M/N \setminus \{ N \}$ and $g$ is $C$-conjugate to every element in $gN$ where $C/N = \Centralizer_{G/N} (gN)$. 
\end{enumerate}
\end{thm}

\begin{proof}
Suppose $g^G = M \setminus N$. 
Then we have $|G:\Centralizer_G (g)| = |M \setminus N| = |M| - |N|$, which is (3).  Conversely, suppose $|G:\Centralizer_G(g)| = |M| - |N|$.  This implies that $|g^G| = |M| - |N|$.  Since $M$ and $N$ are normal in $G$ and $g \in M \setminus N$, we have $g^G \subseteq M \setminus N$. But since $|g^G|=|M|-|N|=|M\setminus N|$, we must have $g^G = M \setminus N$. Therefore, (1) and (3) are equivalent. 

Next, we show that (1) is equivalent to (2). If $g^G=M\setminus N$, then $gN$ is conjugate in $G/N$ to all of the nonidentity elements of $M/N$.  By Lemma \ref{min standard conj}, we know that all the characters in $\irr{G/N}$ are constant on $M/N \setminus \{ N \}$.  Notice that $g$ is conjugate to every element of $gN$.  Thus, the results of \cite{flat} apply.  By Lemma 3 of that paper, every character in $\irr {G \mid N}$ vanishes on $g$.  This implies (2).   Conversely, suppose (2).  Since every character in $\irr {G \mid N}$ vanishes on $g$, we know from Lemma 3 of \cite{flat} that $g$ is conjugate to all of $gN$.  On the other hand, since every character in $\irr {G/N}$ is constant on $M/N \setminus \{ N \}$, we apply Lemma \ref{min standard conj} to see that $gN$ is conjugate to all of the cosets in $M/N \setminus \{ N \}$.  Combining these two results, we obtain the conclusion that $g$ is conjugate to all of the elements in $M \setminus N$, and this proves (1).	

Finally, we show that (1) is also equivalent to (4).  Again, $gN$ is conjugate in $G/N$ to all of the nonidentity elements of $M/N$.  Also, if $h \in gN$, then there exists $x \in G$ so that $g^x = h$.  It follows that $(gN)^x = g^xN = hN = gN$ and so, $x \in C$.  This implies (4).   Suppose (4).  Thus, we know that $gN$ is conjugate in $G/N$ to all of the nonidentity elements of $M/N$.   Combining this with the fact that $g$ is $C$-conjugate to every element in $gN$, we see that $g$ is conjugate to all of $M \setminus N$.  This is Condition (1).
\end{proof}

Observe that if $G$ is any group having a normal subgroup $M$ of order $2$ then, setting $N$ to be the trivial subgroup and $g$ to be the nonidentity element of $M$, we have that $M$, $N$ and $g$ satisfy Condition (1) (and hence all the conditions) of Theorem \ref{standard conj}. The other conditions are also easily checked in this case. In fact, since $M \setminus \{ 1 \} = \{ g \}$, every character in $\irr {G/N}=\irr G$ is automatically constant on $M \setminus \{ 1 \}$; also, as $\irr {G \mid N}$ is empty, the condition that every character in this set vanishes on $g$ is vacuously met, so Condition (2) holds as well.  Note that $g$ will be central in $G$ in this case, so $\Centralizer_G (g) = G$ and $|G:\Centralizer_G (g)| = 1 = |M| - |\{ 1 \}|$, and Condition (3) holds.  Finally, Condition (4) is also obviously met in this situation. 

Next, we are going to obtain some structural results in this situation.  We saw in Lemma~\ref{min standard conj} that $M/N$ will be an elementary abelian $p$-group for some prime $p$.  Surprisingly perhaps, we show that $g$ itself must have $p$-power order.  Since all of the elements of $M \setminus N$ are conjugate, this implies that all of the elements in $M \setminus N$ have $p$-power order.  This is a situation that is studied by the first author in \cite{Lewis}.  We describe one of the situations that arise in this case in the next paragraph.

As part of the next result, we need to discuss Frobenius-Wielandt triples.  In \cite{wielandt}, H. Wielandt considers a group $G$ that has a proper, nontrivial subgroup $H$ and a normal subgroup $L$ of $H$ that is proper in $H$ for which $H \cap H^g \le L$ for all $g \in G \setminus H$.  (Sometimes $L$ is chosen to be minimal among all subgroups that satisfy these conditions, and if we need to add this assumption, we will.  However, in general, we will allow any subgroup $L$ that satisfies the conditions.)  A. Espuelas in \cite{espuelas} calls a group $G$ with such subgroups a {\it Frobenius-Wielandt} group.  However, we will follow \cite{burkett} and say that $(G,H,L)$ is a {\it Frobenius-Wielandt triple}.  When $L = 1$, it is easy to see that $H$ is a Frobenius complement; so this is a generalization of Frobenius groups.  In \cite{wielandt}, Wielandt proves that there is a unique normal subgroup $N$ determined by $H$ and $L$ so that $G = NH$ and $N \cap H = L$.  In fact, the subgroup $N$ satisfies $N = G \setminus \cup_{g \in G}(H \setminus L)^g$, and $N$ is called the {\it Frobenius-Wielandt kernel}.  Note when $L = 1$ that $N$ is the usual Frobenius kernel.  Wielandt also proves in that paper that $(|G:H|,|H:L|) = 1$.  We note that Section 4 of \cite{KnSch} gives a different proof of Wielandt's results.

We will also prove that $M$ is solvable and has a normal $p$-complement.  To see that $G$ does not need to be solvable, we use the observation above that if $M$ has order $2$ and $g$ is the nonidentity element of $M$, then $M$, the trivial subgroup, and $g$ will satisfy our hypotheses.  Notice that this situation occurs when $G = {\rm SL}_2 (q)$ for any odd prime power $q$, and so long as $q \ge 5$, we know that ${\rm SL}_2 (q)$ is not solvable.  This next theorem includes Theorem \ref{intro-standard res} from the Introduction.

\begin{thm}\label{standard res}
Let $G$ be a group, let $N < M$ be normal subgroups of $G$, and suppose that $g \in M \setminus N$ is conjugate to every element in $M \setminus N$.   Then there is a prime $p$ so that the following are true:
\begin{enumerate}[{\rm (a)}] 
\item $g$ has $p$-power order; so every element in $M \setminus N$ has $p$-power order.
\item  Either $M$ is a $p$-group or $M$ is a Frobenius-Wielandt group with Frobenius-Wielandt kernel $N$.	
\item $M$ is solvable and has a normal $p$-complement.
\end{enumerate}
\end{thm}

\begin{proof}
By Lemma \ref{min standard conj}, we know that $M/N$ is a $p$-group for some prime $p$.  Let $P$ be a Sylow $p$-subgroup of $M$.  Observe that $M \le NP$.  On the other hand, since $N$ and $P$ are contained in $M$, we have $NP \le M$; so $M = NP$, and in particular, $P$ is not contained in $N$.  Hence, $M \setminus N$ contains some element of $P$.  Since $g$ is conjugate to every element in $M \setminus N$, it follows that $g$ is conjugate to an element in $P$, and thus, $g$ (and therefore every element in $M \setminus N$) has $p$-power order.  This proves (a). 

By Theorem 1.1 in \cite{primepower}, we see that $M$ is either a $p$-group or a Frobenius-Wielandt group with Frobenius-Wielandt kernel $N$.  This is (b).

Consider a character $\psi \in \irr {G \mid N}$.  By Theorem \ref{standard conj} (2), we know that $\psi$ vanishes on $g$.  Since $g$ has $p$-power order and $\psi$ vanishes on $g$, we see using Remark 4.1 of \cite{DPSS} that $p$ must divide $\psi (1)$.   Observe that ${\rm cd} (G \mid N')$ is a subset of ${\rm cd} (G \mid N)$ and so, $p$ divides every degree in ${\rm cd} (G \mid N')$.  Using a Theorem of Berkovich (Theorem D of \cite{IK}), $N$ is solvable and has a normal $p$-complement, and since $M/N$ is an abelian $p$-group, we conclude that $M$ is solvable and has a normal $p$-complement.  This is (c).
\end{proof}

We next show that the only $p$-groups for which such a conjugacy class can occur is when $p = 2$.  This next result includes Theorem \ref{intro-conj p groups} from the Introduction.  

\begin{cor} \label{conj p groups}
Let $G$ be a $p$-group for some prime $p$, and let $N < M$ be normal subgroups of $G$ with $g \in M \setminus N$ so that $g^G = M \setminus N$.  Then $p = 2$, $|M:N| = 2$, and $g^G = gN$.
\end{cor}

\begin{proof}
Let $|M| = p^a$ and $|N| = p^b$ for integers $a > 0$ and $b \ge 0$.  If $M$ is not contained in $G'$, then the conclusion follows from Lemma \ref{lem:M le G}.  Thus, we may assume $M \le G'$.  By Theorem \ref{standard conj} (3), we know that $|G:\Centralizer_G (g)| = p^a - p^b = p^b (p^{a-b} - 1)$.  This implies that $p^{a-b} - 1$ must equal $1$ since it must be a power of $p$.  However, the only way $p^{a-b} - 1$ can equal $1$ is if $p = 2$ and $a-b = 1$ which implies $a = b+1$, and so, $|M:N| = 2$.  Since $|M:N| = 2$, we know that $g^G = M \setminus N = gN$, and the result is proved.
\end{proof}

Thus, for $2$-groups, we see that we can simplify the character condition that is equivalent for the conjugacy class to be the difference of two normal subgroups.

\begin{cor} \label{conj 2 groups}
Let $G$ be a $2$-group, let $N < M$ be normal subgroups of $G$ so that $M/N$ is minimal normal in $G/N$, and let $g \in M \setminus N$.  Then the following are equivalent:
\begin{enumerate}[$(1)$]
\item $g^G = M \setminus N$. 
\item Every character in $\irr {G/N}$ is constant on $M/N \setminus \{ N \}$ and every character in $\irr {G \mid N}$ vanishes on $g$.
\item Every character in $\irr {G \mid N}$ vanishes on $g$.
\end{enumerate}
\end{cor}

\begin{proof}
We know from Theorem \ref{standard conj} that (1) and (2) are equivalent.  Observe that if (2) occurs then obviously (3) occurs.  Thus, we assume (3).  Since $M/N$ is minimal normal, we have $|M:N| = 2$, and so, $M \setminus N = gN$.  By Lemma 3 of \cite{flat}, we know that the condition that all characters in $\irr{G \mid N}$ vanish on $g$ imply that $g$ is conjugate to all of $gN$.  Hence, we have that $g$ is conjugate to all of $M \setminus N$, and (1) holds.
\end{proof}

Notice that if $G$ is a $2$-group, $M$ is a minimal normal subgroup, and $g$ is the nonidentity element of $M$, then $\{ g \} = M  \setminus \{ 1 \}$ is a conjugacy class of $G$.  Thus, this situation occurs in every $2$-group with some frequency.  Notice that if $G$ is an extra-special $2$-group and $Z = \Center (G)$, $Z \le M  < G$ so that $|M :Z| = 2$, and $g \in M \setminus Z$, then $g^G = gZ = M  \setminus Z$.  So, this situation also occurs when $N > 1$. 

We now work to prove Theorems \ref{intro-mutual} and \ref{intro-mutual 1}.  Before we do this, we need the following refinement of Gagola's theorem.  Notice that by Lemma \ref{min standard conj} the subgroup $M$ in this next theorem is necessarily minimal normal.  On the other hand, by Theorem \ref{gagola}, we know that since $G$ is a Gagola group, that $G$ has a unique minimal normal subgroup and there is a unique irreducible character that satisfies the Gagola property.  This implies that the $M$ and $\chi$ will be uniquely defined in all parts of this theorem.  Similarly, $g$ is uniquely determined up to conjugacy.

\begin{thm}\label{gagola 2} 
Let $G$ be a group with $|G|>2$. Then the following are equivalent:
\begin{enumerate}[$(1)$]
\item There exists a normal subgroup $M$ so that $1 < M \le G$ with $g\in M$ such that $g^G = M \setminus \{ 1 \}$, and $|\irr {G \mid M}| = 1$.
\item  $G$ is a Gagola group with Gagola character $\chi$.
\item There exists a normal subgroup $M$ so that $1 < M \le G$ and a unique character $\chi$ so that $\irr {G \mid M} =\{ \chi \}$. 
\end{enumerate}
If $G$ and $\chi$ satisfy these conditions, then $\chi$  does not vanish on $M \setminus \{ 1 \}$. 
\end{thm}

\begin{proof}
Suppose (1) occurs, and let $\chi$ be the unique character in $\irr {G \mid M}$.  Let $\rho$ be the regular character for $G$, let $\rho^*$ be the regular character for $G/M$ inflated to $G$, and observe that $\rho = \rho^* + \chi (1) \chi$, so $\chi = (\rho - \rho^*)/\chi (1)$.  If $g \in G \setminus M$, then both $\rho$ and $\rho^*$ vanish at $g$, so $\chi (g) = 0$.  Thus, $\chi$ vanishes on every nonidentity conjugacy class except $g^G$.  This proves (2).
	
Assume (2) occurs.  By Theorem \ref{gagola}, we know that $\chi$ is the unique faithful irreducible character of $G$ and $\chi$ does not vanish on $M$ where $M$ is the unique minimal normal subgroup of $G$ (see (b), (c), and (e) of Theorem \ref{gagola}).  On the other hand, any character in $\irr {G \mid M}$ will be faithful, so $\chi$ is the unique character in $\irr {G \mid M}$.  This proves (3).   

Suppose (3) occurs, so $\irr{G \mid M}= \{\chi\}$.  Let $\rho$ be the regular character for $G$, let $\rho^*$ be the regular character for $G/M$ inflated to $G$, and observe that $\rho = \rho^* + \chi (1) \chi$, so $\chi = (\rho - \rho^*)/\chi (1)$.  If $g \in M \setminus \{ 1 \}$, then $\rho (g) = 0$ and $\rho^* (g) =  |G/M|$, so $\chi(g) = -|G/M|/\chi (1)$. In particular, $\chi$ is constant on $M\setminus\{1\}$. On the other hand, all of the irreducible characters other than $\chi$ have $M$ in their kernel. It follows that all the irreducible characters of $G$ are constant on $M\setminus\{1\}$.  This implies that all of the nonidentity elements of $M$ are conjugate, which is (1).
\end{proof}

This next theorem proves more than Theorem \ref{intro-mutual}, under the assumption $|G/N|>2$.  As with Theorem \ref{gagola 2}, we note that in this next theorem, we see from Lemma \ref{min standard conj} that $M/N$ is a minimal normal subgroup of $G/N$ and since we obtain that $G/N$ is a Gagola group, we see from Theorem \ref{gagola} that $G/N$ has a unique minimal normal subgroup and a unique Gagola character.  Thus, $M$ and $\chi$ are uniquely determined in all parts of the theorem and $g$ is unique up to conjugacy.  

\begin{thm} \label{mutual 1}
Let $G$ be a group, and let $N$ be a normal subgroup of $G$ so that $|G/N|>2$. Then the following are equivalent:
\begin{enumerate}[$(1)$]
\item There is a normal subgroup $M$ of $G$ with $N < M$, and an element $g \in M \setminus N$, such that $g^G = M \setminus N$ and $|\irr {G/N \mid M/N}| = 1$. 
\item There is an element $g \in G \setminus N$, and $\chi\in\irr G$ with $N\leq\ker\chi$, such that $\chi$ vanishes on all of the $G$-conjugacy classes in $G \setminus N$ except $g^G$, and every character in $\irr {G \mid N}$ vanishes on $g$.
\item There is a normal subgroup $M$ of $G$ with $N < M$, and a unique $\chi\in\irr G$ that does not vanish on $M \setminus N$ and does not have the elements of $M \setminus N$ in its kernel.
\item $G/N$ is a Gagola group with Gagola character $\chi$. Also, $g$ is $C$-conjugate to every element in $gN$ where $C/N = \Centralizer_{G/N} (gN)$ and $gN$ is a representative of the nontrivial $G/N$-conjugacy class $\chi$ does not vanish on.
\item There is a normal subgroup $M$ of $G$ with $N < M$, and an element $g \in M \setminus N$, such that $|G:\Centralizer_G (g)| = |M| - |N|$ and $|\irr {G/N \mid M/N}| = 1$.  
\end{enumerate}
Furthermore, when the above conditions hold, we have $|{\rm C}_G (g)| = |{\rm C}_{G/N} (gN)|$.
\end{thm}

\begin{proof}
Observe that the equivalence of (1) with (5) follows from  the equivalence of (1) and (3) in Theorem \ref{standard conj}.
Also, (1) implying (4) follows from from (1) implying (2) of Theorem \ref{gagola 2} and (1) implying (4) of Theorem \ref{standard conj}.  The implication of (4) implying (1) is Theorem \ref{gagola} with Theorem \ref{standard conj}.

	
Suppose now (4).  By Theorem \ref{gagola}, we see that $\chi$ is the unique character in $\irr {G/N}$ that does not contain $M/N$ in its kernel where $M/N$ is the unique minimal normal subgroup of $G/N$.   Since we know (4) is equivalent to (1) and (1) is equivalent to (5), we have $|G:\Centralizer_G (g)| = |M| - |N|$, so $\displaystyle |\Centralizer_G (g)| = \frac {|G|}{|M| - |N|}$.   Since $G/N$ is a Gagola group, we have $|G/N:\Centralizer_{G/N} (gN)| = |M/N| - 1$, so $\displaystyle |\Centralizer_{G/N} (gN)| = \frac {|G/N|}{|M/N|-1}$.  Multiplying the numerator and denominator by $|N|$ yields $\displaystyle \frac{|G/N|}{|M/N| -1} = \frac {|G|}{|M| - |N|}$.  We obtain $|\Centralizer_G (g)| = |\Centralizer_{G/N} (gN)|$ (which gives the last claim in our statement).  By the Second Orthogonality Relation (Theorem 2.18 of \cite{Isaacs}), 
we have
$$|\Centralizer_{G/N} (gN)| = \sum_{\phi \in \irr {G/N}} |\phi (gN)|^2 {\rm ~ and ~} |\Centralizer_G (g)| = \sum_{\gamma \in \irr G} |\gamma (g)|^2.$$
We see that 
$$|\Centralizer_G (g)| = \!\!\!\!\sum_{\gamma \in \irr G} \!\!\!\!|\gamma (g)|^2 = \!\!\!\!\sum_{\phi \in \irr {G/N}} \!\!\!\!|\phi (g)|^2 + \!\!\!\!\sum_{\psi \in \irr {G \mid N}} \!\!\!\!|\psi (g)|^2 = |\Centralizer_{G/N} (gN)| + \!\!\!\!\sum_{\psi \in \irr {G \mid N}} \!\!\!\!|\psi (g)|^2.$$
We deduce that $\displaystyle \sum_{\psi \in \irr {G \mid N}} \!\!\!\!|\psi (g)|^2 = 0$.  Since each $|\psi (g)|^2$ is a nonnegative real number, all these values must be zero.  Hence, we have that every irreducible character of $G$ other than $\chi$ either vanishes on $M \setminus N$ or has $M \setminus N$ in its kernel.  This is conclusion (3).
	
Next, suppose (3). Since, by the Second Orthogonality Relation, $\irr {G/N}$ must contain an irreducible character that neither vanishes on $gN$ nor has $gN$ in its kernel, we deduce that $\chi \in \irr {G/N}$.  Note that any other character $\psi$ in $\irr {G/N \mid M/N}$ will vanish on $M/N \setminus \{ N \}$. Since $\psi(1)\neq 0$, it follows that $\psi_{M/N}$ is a multiple of the regular character of $M/N$.  This implies that $\psi_M$ has both $1_M$ and all of the nonprincipal characters in $\irr {M/N}$ as irreducible constituents.  Now, we see that $1_M$ is conjugate to the nonprincipal character in $\irr {M/N}$ and this is a contradiction.  Thus, we see that $\chi$ is the only character that lies in $\irr {G/N \mid M/N}$, and hence, $\irr {G/N \mid M/N} = \{ \chi \}$. Suppose $\theta \in \irr {M/N}$ is nonprincipal.  Then $\theta^G$ is a multiple of $\chi$.  Since $\theta^G$ vanishes off of $M/N$, we deduce that $\chi$ vanishes off of $M/N$ as well.  Viewed as a character of $G$, we see that $\chi$ vanishes on $G \setminus M$.  By Lemmas \ref{Camina pair-equiv} and \ref{gagola pairs}, we see that $G/N$ has the Gagola character $\chi$. Furthermore, suppose $\psi \in \irr{G \mid N}$.  We know from (3) that either $\psi$ has $M \setminus N$ in its kernel or it vanishes on $M \setminus N$.  However, if it has $M \setminus N$ in its kernel, then it has $\langle M \setminus N \rangle$ in its kernel, and it is not difficult to see that $\langle M \setminus N \rangle = M$.  This implies that $\psi$ would have $N$ in its kernel.  Since we are assuming $N$ is not in the kernel of $\psi$, it follows that $\psi (g) = 0$.  Hence (2) holds.
	
Finally, suppose (2).  We see that $\chi$ is an irreducible character of $G/N$ that vanishes on all the nontrivial  conjugacy classes of $G/N$ except the class of $gN$, hence $\chi$ is a Gagola character of $G/N$ and thus, it has unique minimal normal subgroup of $G/N$ which we denote by $M/N$. It follows that $\chi$ is the only character lying in $\irr {G/N \mid M/N}$.  Using the Second Orthogonality Relation and the fact that the characters in $\irr {G \mid N}$ vanish on $g$, again we see that $|\Centralizer_{G/N} (gN)| = |\Centralizer_G (g)|$.  Since $G/N$ is a Gagola group, we have 
$$|G/N:\Centralizer_{G/N} (gN)| = |M:N| - 1.$$  
Now, 
\begin{align*}  |G:\Centralizer_G(g)| &= \frac {|G|}{|\Centralizer_G(g)|} \frac {|N|}{|N|} = \frac {|G/N|}{|\Centralizer_{G/N} (gN)|} |N| = \\ 
	&= |G/N:\Centralizer_{G/N} (gN)| |N| = (|M:N| - 1)|N| = |M| - |N|. 
\end{align*}   
This implies (5).
\end{proof}

It may be worth observing that, whenever the equivalent conditions of the above statement are satisfied, the normal subgroup $M$ which appears in conditions (1), (3) and (5) is necessarily a proper subgroup of $G$. This follows from the fact that all the nontrivial elements of $M/N$ are conjugate in $G/N$, and if $M=G$ this would imply $|G/N|=2$, against the assumptions.

To conclude the proof of Theorem~\ref{intro-mutual}, we have to consider the case when $G/N$ has order $2$ (hence, $M=G$). In this situation, it turns out that the three conditions of that theorem are all equivalent to $G$ being either of order $2$ or a Frobenius group with complements of order $2$. The following proposition should be compared with Proposition~2.1 of \cite{qsy}.

\begin{prop}\label{degenerate}
Let $G$ be a group, $N$ a subgroup of $G$ with $|G:N|=2$, and $g\in G\setminus N$. Then the following are equivalent:
\begin{enumerate}[$(1)$]
\item $g^G = G \setminus N$.
\item Every character in $\irr {G \mid N}$ vanishes on $g$.
\item There is a unique nontrivial irreducible character $\chi$ of $G$ that does not vanish on $G \setminus N$.
\item $|\Centralizer_G (g)| = 2$. 
\item Either $N$ is trivial (and $|G|=2$), or $G$ is a Frobenius group with kernel $N$ and $\langle g\rangle$ as a complement of order $2$. 
\end{enumerate}
\end{prop}

\begin{proof} The condition $g^G=G\setminus N$ yields that $g^G$ contains half of the elements of $G$, which in turn implies $|{\rm C}_G(g)|=2$. The same conclusion is implied also by Condition (2): in fact, if all the
characters in $\irr {G \mid N}$ vanish on $g$, then the column in the character table of $G$ corresponding to $g$ consists of all zeros except the values taken on $g$ by the two irreducible characters of $G/N$ (namely, $1$ and $-1$). An application of the Second Orthogonality Relation yields now $|{\rm C}_G(g)|=2$. Similarly, Condition (3) implies that all the irreducible characters of $G$ vanish on $g$ except the trivial one and the nontrivial character of $G/N$, and again we get $|{\rm C}_G(g)|=2$ via the Second Orthogonality Relation.
Our conclusion so far is that each of (1), (2) and (3) implies (4).

Also, if Condition (4) holds, then $g$ has order $2$ and we get $G=N\langle g\rangle$ with $g$ not centralizing any nontrivial element of $N$; in other words, either $N$ is trivial and $|G|=2$, or $G$ is a Frobenius group with kernel $N$ and $\langle g\rangle$ as a complement, so we have Condition~(5).

Since it is easy to see that both a group $G$ of order $2$ and a Frobenius group $G$ with kernel $N$ and $\langle g\rangle$ as a complement of order $2$ satisfy all conditions (1), (2), (3) and (4), the proof is complete.
\end{proof}

The following result includes Theorem \ref{intro-mutual 1}.


\begin{thm} \label{mutual 2}
Assume Hypothesis $\ref{Hypothesis 1}$.   Then there is a prime $p$ so that the following are true:
\begin{enumerate}[{\rm (a)}]
\item  $M/N$ is an elementary abelian $p$-group.  Also, $gN$ is conjugate in $G/N$ to every coset in $M/N \setminus \{ N \}$.
\item $g$ has $p$-power order; so every element in $M \setminus N$ has $p$-power order.
\item  Either $M$ is a $p$-group or $M$ is a Frobenius-Wielandt group with $N$ the Frobenius-Wielandt kernel.
\item $M$ is solvable and has a normal $p$-complement.
\item $\displaystyle \chi (x) = \left\{ \begin{array}{ll}|G:M|_{p'} \sqrt {|G:M|_p} & {\rm if~~} x \in N \cr
	\displaystyle -\frac {|G:M|}{\chi (1)} = -\sqrt{|G:M|_p} & {\rm if~~} x \in g^G = M \setminus N \cr 
	0 & {\rm if~~} x \in G \setminus M. \end{array} \right.$ 
\item $\displaystyle |\Centralizer_G (g)| = \frac {|G:M| (|G:M| + \chi (1)^2)}{\chi (1)^2} = |\Centralizer_{G/N} (gN)| = |G/N|_p$.  In particular, $\Centralizer_G (g)$ is a $p$-group.
\end{enumerate}
Furthermore, if $|G/N|>2$, then we also have $M \le G'$.
\end{thm}

\begin{proof}
Observe that Hypothesis \ref{Hypothesis 1} is Condition (1) of Theorem \ref{mutual 1}.  By Hypothesis \ref{Hypothesis 1}, we have that every element of $M \setminus N$ is conjugate to $g$, and so, we see that $gN$ is conjugate to all of the cosets in $M/N \setminus \{ N \}$.  By Lemma \ref{lem:M le G} (a), we have that either $M \le G'$ or $N = M \cap G'$ and $|M/N|=2$.   If $N = M \cap G'$, then $G'M/N = G'/N \times M/N$.  We see that $\irr {G'M/N \mid M/N} = \{ \theta \times \phi \mid  \theta \in \irr {G'/N}, \phi \in \irr {M/N}, \phi \ne 1 \}$.  Note that each $1 \times \phi$ and $\theta \times \phi$ will lie in different $G$-orbits when $\theta$ is not principal. This yields the contradiction that  $\irr {G/N \mid M/N}$ contains at least two characters \emph{unless $G'/N$ is trivial}. Thus, if $G'\neq N$, we get $M \le G'$. On the other hand, if $G'=N$, then $G/N$ is abelian, and our assumption $|\irr{G/N\mid M/N}|=1$ yields $|G/M|=|G/N|-1$, which easily implies $G=M$ (and $|G/N|=2$).

The fact that $M/N$ is an elementary abelian $p$-group is Lemma \ref{min standard conj} (a). 
%
We now obtain (b), (c), and (d) via Theorem \ref{standard res}.
	
	
Since $g \in M$, we have $gM = M$, and by the Second Orthogonality Relation in $G/M$, we see that $\displaystyle \sum_{\mu \in \irr {G/M}} \mu (g) \mu (1) = |G/M|$.    Applying Second Orthogonality Relation 
in $G/N$ where $gN$ is not conjugate to $N$, we have using the uniqueness of $\chi$ in Theorem \ref{mutual 1}:
$$0 = \sum_{\nu \in \irr {G/N}} \nu(g) \nu (1) = \sum_{\mu \in \irr {G/M}} \mu (g) \mu (1) + \chi (g) \chi (1) = |G/M| + \chi (g)\chi (1).$$
Solving for $\chi (g)$, we obtain $\chi (g) = -|G:M|/\chi (1)$.  
	
We know from Theorem \ref{mutual 1} that ${\rm C}_G (g)|=|\Centralizer_{G/N} (gN)|$, and, by Corollary 2.3 of \cite{Gagola}, we see that $|\Centralizer_{G/N} (gN)| = |G/N|_p$.  Notice that $\Centralizer_G (g)$ has $p$-power order, and so, $\Centralizer_G (g)$ is a $p$-group.  The computation of $\chi (1)$ is done in the last displayed equation on page 382 of \cite{Gagola}.  Note that the previous displayed equation on page 382 of \cite{Gagola} gives the fact that, for $x \in M$, $\chi (x) = -\sqrt {|G:M|_p}$.   We now have all of the parts of (e).
	
Finally, if we replace $\chi (g)$ by $-|G:M|/\chi (1)$ in $|\Centralizer_G (g)| = |G:M| + |\chi (g)|^2$, we obtain 
$$|\Centralizer_G (g)| = |G:M| + \left(\frac {-|G:M|}{\chi (1)}\right)^2 = \frac {|G:M|(|G:M| + \chi(1)^2)}{\chi (1)^2}.$$
This finishes conclusion (f). 
\end{proof}


We now outline properties of groups with a Gagola character.

\begin{cor}\label{gagola 1}
Suppose $G$ is a group with a Gagola character $\chi$, let $g$ be a representative of the nonidentity conjugacy class that $\chi$ does not vanish on, and take $M = \langle g^G \rangle$.  Then for some prime $p$:
\begin{enumerate}[{\rm (a)}]
\item $\displaystyle \chi (x) = \left\{ \begin{array}{ll} |G:M|_{p'} \sqrt {|G:M|_p} & {\rm if~~} x = 1 \cr
	\displaystyle -\frac {|G:M|}{\chi (1)} = -\sqrt{|G:M|_p} & {\rm if~~} x \in g^G = M \setminus \{ 1 \} \cr 
			0 & {\rm if~~} x \in G \setminus M. \end{array} \right.$
\item $\displaystyle |\Centralizer_G (g)| = \frac {|G:M| (|G:M| + \chi (1)^2)}{\chi (1)^2} = |G|_p$. 
\end{enumerate} 
\end{cor}

\section{Proofs of Theorems \ref{th:character} and \ref{th:class} } \label{occurs 1}

We are ready to prove Theorem~\ref{th:character}. The following is an extended form of it.

\begin{thm} \label{detailedThm7}
Let $G$ be a group, and let $\chi$ be a nonlinear irreducible character of $G$.  Then the following properties are equivalent.
\begin{enumerate}[$(1)$]
\item There exists an element $g \in G \setminus \Center (G)$ such that $\chi$ vanishes on all noncentral conjugacy classes of $G$ except $g^G$.
\item We have $\Center (G) \leq \kernel (\chi)$, and there exists a noncentral element $g$ of $G$ such that, for every character $\psi \in \irr G$ with $\psi \ne \chi$, either $g \in \kernel (\psi)$ when $\psi \in \irr {G/\Center (G)}$ or otherwise $\psi (g) = 0$.
\item We have $\Center (G) \leq \kernel (\chi)$, and $\chi$ may be viewed as a character of $G/\Center (G)$ that vanishes on all but two conjugacy classes of $G/\Center (G)$, that is, $G/\Center (G)$ is a Gagola group.  In addition, if $\psi \in \irr {G \mid \Center (G)}$, then $\psi (g) = 0$ when $g \Center (G)$ is an element of the nontrivial conjugacy class of $G/\Center (G)$ that $\chi$ does not vanish on.  
\end{enumerate} 
Furthermore, if the above properties hold, then we have $\Center (G) = \kernel (\chi) < G'$, the value $\chi (g)$ is $-|\Center (G)|\chi(1)/|g^G| \in \mathbb{Z}$, and $|\Centralizer_G (g)| = |\Centralizer_{G/\Center (G)} (g\Center (G))|$ is a $p$-power for a suitable prime $p$. In particular, ${\rm Z}(G)$ is a $p$-group. Finally, if $M/{\Center(G)}$ denotes the unique minimal normal subgroup of $G/{\Center(G)}$, we have that $M\setminus{\Center(G)}$ is a conjugacy class of $G$.
\end{thm}

\begin{proof}
Set $Z = \Center (G)$.  We first assume (1), that is, $\chi (x) = 0$ for all elements $x \in G \setminus (g^G \cup Z)$.  We have $\chi_Z = \chi(1) \lambda$ for some character $\lambda \in \irr Z$. Since $\chi \neq 1_G$, by the First Orthogonality Relation (Corollary 2.14 of \cite{Isaacs}), we have 
$$0 =|G| (\chi, 1_G) = \sum_{x \in G} \chi(x) = \sum_{x \in Z} \chi(x) + |g^G| \chi (g) = \chi (1) \sum_{x \in Z} \lambda (x) + |g^G| \chi (g).$$ 
Since $\chi (g) \neq 0$, the first sum cannot be zero so $\lambda = 1_Z$, which implies that $Z \leq \kernel (\chi)$. Moreover, we get $\chi(g)=-|Z|\chi(1)/|g^G|$, which is in $\mathbb{Z}$ because it is a rational algebraic integer.
	
Consider a character $\psi$ so that $\chi \neq \psi \in \irr G$. By the First Orthogonality Relation again, we have 
\[0 = \sum_{x \in G} \psi (x) \overline{\chi (x)} = \chi (1) \sum_{z \in Z} \psi(z) + |g^G| \psi (g) \chi (g).\] 
If $Z \leq \kernel (\psi)$, then $\sum_{z \in Z} \psi (z) = |Z| \psi(1)$, and hence, substituting the value of $\chi(g)$, we obtain $\psi(g) = \psi(1)$. On the other hand, if $Z \not\leq \kernel (\psi)$, then $\sum_{z \in Z} \psi (z) = 0$ so that $\psi (g) = 0$ and (2) follows.  
	
Next, assume (2).  Let $M = \langle g^G \rangle Z$.  We have $\chi \in \irr {G/Z}$ and if $\psi \in \irr {G/Z}$ with $\psi \ne \chi$, then $M \le \ker {\psi}$.  It follows that $\irr {G/Z \mid M/Z} = \{ \chi \}$.  Let $\rho_M$ be the regular character of $G/M$ and $\rho_Z$ be the regular character of $G/Z$.  Thus, $\rho_M = \sum_{\psi \in \irr {G/M}} \psi(1) \psi$ and $\rho_Z = \sum_{\psi \in \irr {G/M}} \psi (1) \psi + \chi (1) \chi = \rho_M + \chi (1) \chi$.  We know $g \in M$, so $\rho_M (g) = |G:M|$ and $g \not\in Z$, so $\rho_\Center (g) = 0$.  Hence, $0 = |G:M| + \chi (1) \chi (g)$, and so, $\chi (g) = -|G:M|/\chi (1) \ne 0$.  Hence, $\chi$ satisfies Condition (3) of Theorem \ref{gagola 2}.  It follows that $\chi$ is a Gagola character for $G/Z$.  Notice that if $\psi \in \irr {G \mid Z}$, then we must have that $\psi (g) = 0$.  Thus, (3) holds.  Notice that this also yields $\ker {\chi} = Z$.  

Next, assume (3).  Observe that this gives Condition (4) of Theorem \ref{mutual 1}.  Observe that Condition (2) of Theorem \ref{mutual 1} gives Condition (1) of this theorem.

Finally, assume that the above equivalent properties hold.  We have shown above that (3) is essentially (4) of Theorem \ref{mutual 1}.  This implies that $G$ satisfies Hypothesis \ref{Hypothesis 1} with respect to $N=\Center (G)$ and $M$ such that $M/\Center(G)$ is the unique minimal normal subgroup of $G/\Center(G)$. Thus, in particular $M\setminus\Center(G)$ is a conjugacy class of $G$. By Theorem \ref{mutual 2} we have $\Center (G) \le G'$ and $|\Centralizer_G (g)| = |\Centralizer_{G/\Center (G)} (g\Center (G))|$.  Note that since $G/\Center (G)$ is necessarily nonabelian, we obtain $\Center (G) < G'$.  
In Theorem \ref{mutual 2} (f) it is proved that $|\Centralizer_G (g)| = |G:\Center (G)|_p$ is a $p$-power; thus, since ${\rm Z}(G)$ is contained in ${\rm C}_G(g)$, we deduce that ${\rm Z}(G)$ is a $p$-group.  In Theorem \ref{mutual 2} (e) it is proved that $\chi (1) = |G:M|_{p'}\sqrt{|G:M|_p} = \chi (z)$ for all $z \in \Center (G)$.  Also, for $y \in M$, we have $\chi (y) = -|G:M|/\chi (1) = - \sqrt{|G:M|_p}$, and $\chi (x) = 0$ for $x \in G \setminus M$.  This implies that $\Center (G) = \kernel (\chi)$. We have already seen that $\chi(g)=-|Z|\chi(1)/|g^G|$, and the proof is concluded.
\end{proof}

The following theorem includes Theorem \ref{th:class}.

\begin{thm}\label{detailedThm6}
Let $G$ be a group. Assume there exist an element $g \in G\setminus \Center (G)$ and a nonlinear character $\chi \in \irr G$ such that $\chi(g)\neq 0$ and, for every nonlinear irreducible character $\psi \neq \chi$ of $G$, we have $\psi (g) = 0$. Then, setting $K = \kernel(\chi)$, the following conclusions hold.
\begin{enumerate}[{\rm (a)}]
\item $g \in G'$ and $\chi (g) = -|G:G'|/\chi (1)$.
\item $|\Centralizer_G (g)| = |G:G'|(|G:G'| + \chi(1)^2)/\chi(1)^2$.
\item $G' = K \cup g^G$ and $\chi$ vanishes off $G'$.
\item $G'/K$ is the unique minimal normal subgroup of $G/K$.
\item $\chi$ is the unique nonlinear irreducible character of $G/K$. 
\item $G/K$ is either an extra-special $2$-group of order $2^{1+2m}$ or a Frobenius group whose kernel is an elementary abelian $p$-group of order $p^m$ and a Frobenius complement is cyclic of order $p^m-1$, for a suitable prime $p$ and a suitable integer $m\ge 1$.  Moreover, $|\Centralizer_G(g)|=2^{1+2m}$ or $p^m$, respectively, and $g$ is an element of prime-power order.
\item $G$ is solvable.
\end{enumerate}
\end{thm}

\begin{proof}
Let us denote by $\Lin(G)$ the set of linear characters of $G$, and by $X$ the set of nonlinear irreducible characters of $G$ different from $\chi$. Since $g$ is nontrivial and $\phi(g) = 0$ for all $\phi\in X$, it follows from the Second Orthogonality Relation 
that  
\[0 = \sum_{ \phi \in \irr G} \phi(g) \phi(1) = \sum_{ \phi \in \Lin(G)} \phi(g) \phi(1) + \chi(g) \chi(1). \]
If $g \not\in G'$, then  
$$\sum_{ \phi \in \Lin(G)} \phi(g) \phi(1) = 0$$ 
by the Second Orthogonality Relation, since $G' \neq gG' \in G/G'$. But this would imply $\chi (g) = 0$, a contradiction. Hence $g \in G'$, so $ \sum_{ \phi \in \Lin(G)} \phi(g) \phi(1) =|G/G'|$ and the equality above yields $\chi(g)=-|G/G'|/\chi(1)$, proving (a). Note that, if $x$ lies in $K$, then $x$ is not $G$-conjugate to $g$; hence, the Second Orthogonality Relation applied to $x$ and $g$ yields \[0 = \sum_{ \phi \in \Lin(G)} \phi(x) \phi(1) + \chi(1) \chi(g).\] Now, if $x\not\in G'$, then we get the contradiction $0=0+\chi(1)\chi(g)$, thus we conclude that $K\subseteq G'$. Also, if $y\in G'$ is not $G$-conjugate to $g$ then, again by the Second Orthogonality Relation applied to $y$ and $g$, we see that $0=|G/G'|+\chi(y)\chi(g)=|G/G'|-(|G/G'|\chi(y)/\chi(1))$; it easily follows that $\chi(y)=\chi(1)$, hence $y\in K$. We conclude that $g^G=G'\setminus K$, that is, $G'=K\cup g^G$. Now, Condition (3) of Theorem \ref{mutual 1} is satisfied (note that $|G/K|>2$, as $\chi$ is nonlinear), hence Condition (1) of that theorem is satisfied as well. So, Hypothesis~\ref{Hypothesis 1} holds, and (b) follows from (f) of Theorem \ref{mutual 2}. The second part of (c) and (d) are consequences of the fact that, by Theorem \ref{mutual 1} (4), $\chi$ is a Gagola character of $G/K$. As for (e), this is implied by Theorem \ref{mutual 1} (1) taking into account that the subgroup $M/N$ in our case is $(G/K)'$. 
Since $G/K$ has a unique nonlinear irreducible character, we may apply a theorem due to G. Seitz (\cite{Seitz}) to see that $G/K$ is either an extra-special $2$-group of order $2^{1+2m}$, or a Frobenius group whose kernel is an elementary abelian $p$-group of order $p^m> 2$ and a Frobenius complement is cyclic of order $p^m-1$, for a suitable prime $p$ and a suitable integer $m\ge 1$. (We will see that we could also appeal to groups with a Gagola character here where the $G'$ is the unique minimal normal subgroup.) 
	
If $G/K$ is an extra-special $2$-group, then $|G:G'|=2^{2m}$ and $\chi(1)=2^m$, whereas if $G/K$ is a Frobenius group, then $|G:G'| = \chi(1) = p^m - 1$. By (b), we have that $|\Centralizer_G(g)| = 2^{1+2m}$ or $p^m$. In particular, $g$ is a $2$-element or a $p$-element. Thus $g$ is an $r$-element for some prime $r\in \{ 2,p \}$, and the proof of (f) is complete.
	
Finally, we claim that $K$ is solvable. Suppose not, and let $L$ be the last term of the derived series of $K$. Then $L=L'\subseteq  G'$.  Let $\varphi$ be an irreducible character of $G$ whose kernel does not contain $L'$. Then $\varphi\neq \chi$ and is nonlinear, so $\varphi (g) = 0$. By Remark~4.1 in \cite{DPSS}, $r$ divides $\varphi (1)$.  By Theorem D in \cite{IK}, $L$ is solvable, which is a contradiction. Therefore, $K$ is solvable and since $G/K$ is solvable, we have that $G$ is solvable as well, proving (g).
\end{proof}

In Theorem \ref{th:class}, note that if ${\rm Ker} (\chi) = 1$, then $G$ is a group where $G'\setminus\{1\}$ consists of a single conjugacy class of $G$. This implies that $G'$ is a minimal normal subgroup of $G$, and it is not difficult to see that the only possibilities are that $G$ is an extra-special $2$-group or $G$ is a Frobenius group where the Frobenius kernel is elementary abelian of order $p^m$ for a prime $p$ and positive integer $m$ so that a Frobenius complement is cyclic of order $p^m - 1$.  In the next several results, we want to study groups that satisfy the hypotheses of Theorem \ref{th:class} (or Theorem \ref{detailedThm6}).  For this reason, we set down these hypotheses:

\medskip\noindent
\begin{hyp}\label{Hypothesis 2}  Let $G$ be a group.  Assume there exists an element $g \in G \setminus \Center (G)$ and a nonlinear character $\chi \in \irr G$ so that $\chi(g)\neq 0$ and, for every nonlinear irreducible character $\psi \ne \chi$ of $G$, we have $\psi (g) = 0$.  Set $K = \kernel (\chi)$.
\end{hyp}
\medskip

We first consider groups satisfying Hypothesis \ref{Hypothesis 2} with $K = 1$.  

\begin{lem} \label{degenerate case}
If $G$ satisfies Hypothesis $\ref{Hypothesis 2}$ with $K = 1$, then either $G$ is an extra-special $2$-group of order $2^{1+2m}$ or a Frobenius group whose Frobenius kernel is an elementary abelian $p$-group of order $p^m$ and a Frobenius complement is cyclic of order $p^m - 1$, where $p$ is a prime and $m\geq 1$ is an integer.
\end{lem}

\begin{proof}
This is Theorem \ref{detailedThm6} (f).
\end{proof}

Thus, we now consider groups that satisfy Hypothesis \ref{Hypothesis 2} with $K > 1$.  For these groups, we show that $\Center (G) \le K$.

\begin{lem}
If $G$ satisfies Hypothesis $\ref{Hypothesis 2}$ with $K > 1$, then $\Center (G) \le K$.
\end{lem}    

\begin{proof}
By Theorem \ref{detailedThm6} (f), we know that $G/K$ is either an extra-special $2$-group or a Frobenius group.  Let $Z/K = \Center (G/K)$, and observe that $\Center (G) \le Z$.  If $G/K$ is a Frobenius group, then $Z/K = 1$, so $Z = K$ and $\Center (G) \le K$.  If $G/K$ is an extra-special $2$-group, then $Z/K = (G/K)' = G'/K$ and $|Z:K| = 2$.  If $\Center (G)$ is not contained in $K$, then  $Z \setminus K$ contains elements in $\Center (G)$ and this contradicts Theorem \ref{detailedThm6} (c) which states that all of the elements in $G' \setminus K$ are conjugate.   Hence, we conclude that $\Center (G) \le K$.
\end{proof}

For groups satisfying Hypothesis \ref{Hypothesis 2}, we can use the results proved in Theorem \ref{mutual 2} to obtain more information on $G'$. 

\begin{thm}\label{Frob-Wiel}
If $G$ satisfies Hypothesis $\ref{Hypothesis 2}$, then either $G'$ is a $p$-group for a prime $p$ or $G'$ is a Frobenius-Wielandt group with $K$ its Frobenius-Wielandt kernel.
\end{thm}

\begin{proof}
This is Theorem \ref{mutual 2} (c).  
\end{proof}

We now consider the consequences of Theorem \ref{Frob-Wiel} in light of the possibilities for $G/K$ from Theorem \ref{detailedThm6}.  We begin by considering the case where $G/K$ is an extra-special $2$-group in Theorem \ref{detailedThm6}.
In this case, we see that $|G':K| = 2$, so either $G'$ is a $2$-group (in which case the whole $G$ is a $2$-group) or $G'$ is a Frobenius-Wielandt group. We will resume the discussion about $2$-groups in Section~\ref{occurs 2}.  
Instead, we now consider the case where $G/K$ is an extra-special $2$-group and $G'$ is a Frobenius-Wielandt group \emph{with the additional assumption that $|G'/K|$ and $|K|$ are coprime}, that is, $K\neq 1$ has odd order.  We will see in the next theorem that $G'$ is a Frobenius group with Frobenius kernel $K$ and whose Frobenius complements have order~$2$.

\begin{thm} \label{2 quo coprime}
Let $G$ be a group:
\begin{enumerate}[{\rm (a)}]
\item If $G$ satisfies Hypothesis $\ref{Hypothesis 2}$ with $G/K$ an extra-special $2$-group and $K$ a nontrivial subgroup of odd order, then $G'$ is a Frobenius group with Frobenius kernel $K$ and Frobenius complements of order $2$. 
\item  Conversely, if $G'$ is a Frobenius group with Frobenius kernel $K$ and Frobenius complements of order~$2$, 
and $G/K$ is an extra-special $2$-group, then $G$ satisfies Hypothesis $\ref{Hypothesis 2}$.
\end{enumerate}
\end{thm}

\begin{proof}
Let $P$ be a Sylow $2$-subgroup of $G$ so that $G = KP$ and $K \cap P = 1$.  Without loss of generality, we may assume that $g \in P$.  By Theorem \ref{detailedThm6} (f), we know that $\Centralizer_G (g) = P$, so $g \in \Center (P)$.  Since we are assuming $P$ is an extra-special $2$-group, this implies that $g$ has order $2$, and $P' = \Center (P) = \langle g \rangle$.  Because $G/K \cong P$, we see that $G' = K \langle g \rangle$.  It follows that $G'$ is a Frobenius group with $K$ as Frobenius kernel and $\langle g \rangle$ as Frobenius complement.  	This gives the first conclusion.

Suppose now that $G'$ is a Frobenius group with Frobenius kernel $K$ and a Frobenius complement that is cyclic of order $2$, and that $G/K$ is an extra-special $2$-group.  We first note that $\Center (G) = 1$.  Now, take $g$ to be an element of $G'$ that lies outside of $K$, and observe that $g$ must have order $2$ and every nonlinear irreducible character of $G'$ will vanish on $g$.  Since $G/K$ is an extra-special $2$-group, we know that $G/K$ has a unique nonlinear irreducible character, which we will denote by $\chi$.  Note that $K = G''$; so every irreducible character of $G'$ that does not have $K$ in its kernel will be nonlinear.  In particular, if $\psi$ is a nonlinear irreducible character of $G$ other than $\chi$, we know that $K$ is not in the kernel of $\psi$; so $K$ is not in the kernel of the irreducible constituents of $\psi_{G'}$.  Hence, the irreducible constituents of $\psi_{G'}$ are not linear.  We have seen that the nonlinear irreducible characters of $G'$ vanish on $g$.  It follow that $\psi$ will vanish on $g$.  Therefore, we have an element $g \in G\setminus{\rm{Z}}(G)$ and a nonlinear character $\chi\in\irr G$ so that all of the other nonlinear irreducible characters of $G$ vanish on $g$.  Thus, $G$ satisfies Hypothesis \ref{Hypothesis 2}. 
\end{proof}

Note that the quaternion group of order $8$ is the only generalized quaternion group that is an extra-special $2$-group.  So the only Frobenius groups that satisfy the conclusion of Theorem \ref{2 quo coprime} are precisely those where a Frobenius complement is the quaternion group of order $8$.  But there are examples of groups satisfying the conclusion of Theorem \ref{2 quo coprime} where $G$ is {\it not} a Frobenius group: the semi-direct product of the dihedral group of order $8$ acting half-transitively on an elementary abelian $3$-group of order $9$ gives an example of this kind. (See \cite{half}.)   

To sum up, for a group $G$ satisfying the assumptions of Theorem~\ref{detailedThm6} in the situation where $G/K$ is an extra-special $2$-group and $G'$ is a Frobenius-Wielandt group, it remains to treat the case when $K$ has even order. We leave the discussion about the detailed structure of such a group as an open problem.

Next, we move to the situation where (still in the context of Theorem~\ref{detailedThm6}) $G/K$ is a Frobenius group whose Frobenius kernel $G'/K$ is an elementary abelian $p$-group for a prime $p$. By Theorem~\ref{Frob-Wiel}, we know that $G'$ is either a $p$-group or a Frobenius-Wielandt group. As for the former case, we leave it for future research; here we only note that, with the extra assumption $K={\rm Z}(G)$, such a group will satisfy the hypotheses of Theorem~\ref{detailedThm7}, and if we assume further $|G':K|=p^2$ and $|K|=p$, then we will be in the situation discussed at the end of Section 6 (see in particular Theorem~\ref{doubly-transtive p^2}).
Moving to the latter case, similarly to Theorem~\ref{2 quo coprime}, in the following result we consider the situation when $|G'/K|$ and $|K|$ are coprime (leaving the non-coprime case for future research).

\begin{thm}
Let $G$ be a group:	
\begin{enumerate}[{\rm (a)}]
\item If $G$ satisfies Hypothesis $\ref{Hypothesis 2}$ with $G/K$ a Frobenius group whose Frobenius kernel $G'/K$ is an elementary abelian $p$-group of order $p^m$ for some prime $p$ and positive integer $m \ge 1$, $G/G'$ is cyclic of order $p^m-1$, and $K$ is a $p'$-group, then $G$ is a $2$-Frobenius group, i.e. $G'$ is a Frobenius group with Frobenius kernel $K$ and in particular, $m = 1$.
\item Conversely, if $G$ is a $2$-Frobenius group where $G'$ is a Frobenius group with Frobenius kernel $K$ and $G/K$ is a Frobenius group with Frobenius kernel $G'/K$ so that $|G':K| = p$ and $G/G'$ is cyclic of order $p-1$, then $G$ satisfies Hypothesis $\ref{Hypothesis 2}$.
\end{enumerate} 
\end{thm}

\begin{proof}
Let $g$ be an element so that $G' \setminus K = g^G$.  Applying Theorem \ref{detailedThm6} (c), we have $G' = K \cup g^G$. Since $g^G=g^K$, we see that $G'$ is a semi-direct product of $K$ with $\langle g\rangle$; it follows that the elementary abelian $p$-group $G/K$ is cyclic, so it has order $p$ 
(that is, $m = 1$).  By Theorem \ref{detailedThm6} (f), we know that $\Centralizer_G (g) = \langle g \rangle$.  It follows that $G' = K \langle g \rangle$ is a Frobenius group with Frobenius kernel $K$ and Frobenius complement $\langle g \rangle$.  Since $G/K$ is a Frobenius group with Frobenius kernel $G'/K$, we conclude that $G$ is a $2$-Frobenius group.  

Conversely, suppose $G$ is a $2$-Frobenius group where $G'$ is a Frobenius group with Frobenius kernel $K$ and $G/K$ is a Frobenius group with Frobenius kernel $G'/K$ so that $|G':K| = p$ and $G/G'$ is cyclic of order $p-1$ for some prime $p$.  We first note that $\Center (G) = 1$  Now, take $g$ to be an element of $G'$ that lies outside of $K$, and observe that $g$ must have order $p$ and every nonlinear irreducible character of $G'$ will vanish on $g$.  Since $G/K$ is a Frobenius group whose Frobenius kernel has order $p$ and a Frobenius complement has order $p - 1$, we know that $G/K$ has a unique nonlinear irreducible character which we will call $\chi$.  Note that $K = G''$; so every irreducible character of $G'$ that does not have $K$ in its kernel will be nonlinear.  In particular, if $\psi$ is a nonlinear irreducible character of $G$ other than $\chi$, we know that $K$ is not in the kernel of $\psi$; so $K$ is not in the kernel of the irreducible constituents of $\psi_{G'}$.  Hence, the irreducible constituents of $\psi_{G'}$ are not linear.  We have seen that the nonlinear irreducible characters of $G'$ vanish on $g$.  It follow that $\psi$ will vanish on $g$.  Therefore, we have an element $g \in G$ and a character $\chi$ so that all of the other nonlinear irreducible characters of $G$ vanish on $g$.  Thus, $G$ satisfies Hypothesis~\ref{Hypothesis 2}.	
\end{proof}



\section{Central extensions of Gagola groups: $p$-groups}\label{occurs 2}

Our aim for the rest of this paper is to describe the groups $H$ that can appear as factor groups $G/{\rm{Z}}(G)$, where $G$ is a group satisfying Condition (1)  of Theorem~\ref{detailedThm7}: $G$ has a nonlinear irreducible character $\chi$ and a noncentral element $g$ such that $\chi$ vanishes on all the noncentral conjugacy classes of $G$ except $g^G$. As shown in that theorem, this is equivalent to the fact that $G/{\rm Z}(G)$ is a Gagola group with Gagola character $\chi$ and every character in $\irr{G\mid{\rm Z}(G)}$ vanishes on $g$ (where the conjugacy class of $g{\rm Z}(G)$ in $G/{\rm Z}(G)$ is the set of nontrivial elements of the unique minimal normal subgroup of $G/{\rm Z}(G)$). Clearly, the groups $H$ as above in the case when ${\rm Z}(G)=1$ are precisely the Gagola groups, so we will assume that ${\rm Z}(G)$ is nontrivial.

In view of the paragraph above, the groups $H$ of our interest are Gagola groups, and we start by considering the case when ${\rm Z}(H)$ is nontrivial. By the definition of a Gagola group, it is immediate to see that ${\rm Z}(H)\neq 1$ if and only if $|{\rm Z}(H)|=2$ (and ${\rm Z}(H)$ is the unique minimal normal subgroup of $H$); in this section we will focus on this case \emph{assuming that $H$ is a $p$-group}, which of course implies that $H$ is a $2$-group. In view of Lemma \ref{gagola pairs}, we also know that $|\irr {G \mid \Center (G)}| = 1$ and $(G,\Center (G))$ is a Camina pair. We note that there has been almost no research studying Camina pairs $(H,\Center (H))$ where $H$ is a $2$-group and $|\Center (H)| = 2$, but Theorem 7.1 of \cite{mac} shows that $|H:\Center (H)|$ is a square in this case. The smallest order for such a group $H$ is obviously $8$ (both the dihedral group and the quaternion group of order $8$ are Gagola groups), and ``small" groups of this kind can be easily determined using {\texttt{GAP}} \cite{GAP} or {\texttt{MAGMA}} \cite{MAGMA} (for instance, there are $7$ Gagola groups of order $32$ up to isomorphism, $75$ Gagola groups of order $128$ and $3095$ Gagola groups of order $512$). As another remark, of course in this situation the group $H$ will have nilpotency class $2$ if and only if it is an extra-special $2$-group. 

To the end of describing central extensions of Gagola groups, the next general lemma will be useful.

\begin{lem} \label{lem: center}
Let $G$ be a group having an irreducible character that vanishes on all but one noncentral conjugacy class of $G$.  If $G = H \times K$, then either $H$ or $K$ is trivial.
\end{lem}

\begin{proof}
Suppose that $H > 1$ and $K > 1$.  We know $\Center (G) = \Center (H) \times \Center (K)$.   Let $\chi$ be the irreducible character that vanishes on all of the conjugacy classes of $G$ except those in the center of $G$ and one noncentral conjugacy class. Suppose $H$ and $K$ are both nontrivial.  We can write $\chi = \theta \times \gamma$ where $\theta \in \irr H$ and $\gamma \in \irr K$.  Let $g$ be a representative of the noncentral conjugacy class that $\chi$ does not vanish on.  We can write $g = (h,k)$ where $h \in H$ and $k \in K$.  We know that $\chi (g) = \theta (h) \gamma(k)$.  Since $\chi (g)$ is not $0$, neither $\theta (h)$ nor $\gamma (k)$ can be $0$.   Now, $\chi ((h,1)) = \theta(h) \gamma (1) \ne 0$ and $\chi ((1,k) )= \theta (1) \gamma(k) \ne 0$.  Since $(h,1)$ and $(1,k)$ are in different conjugacy classes, they cannot both be noncentral, and if they are both central, then $g$ would be central.  Hence, without loss of generality $(h,1)$ is noncentral in $G$ and $(1,k)$ is central in $G$.  Further, $(h,1)$ and $g$ have to be conjugate, so $k = 1$.  Thus, if $k' \in K \setminus \{1\}$, then $\chi ((h,k')) = 0 = \theta (h) \gamma (k')$ and $\gamma (k') = 0$ since $\theta (h) \ne 0$.  It follows that $\gamma$ is a multiple of the regular character of $K$, but this only occurs when $K = 1$.  This proves the lemma.
\end{proof}

As an immediate consequence, we get the following.

\begin{prop}
Let $G$ be a group, and let $H=G/{\rm Z}(G)$. Assume that $G$ has an irreducible character that vanishes on all but one noncentral conjugacy class of $G$, and that $H$ is a $p$-group. Then $H$ is a Gagola $2$-group and $G$ is a $2$-group. 
\end{prop}

\begin{proof} The fact that $H$ is a Gagola $2$-group is explained in the first two paragraphs of this section. In particular, $G/{\rm Z}(G)=H$ being nilpotent, we have that $G$ is nilpotent as well. Now, for any odd prime $p$, the Sylow $p$-subgroup $P$ of $G$ must be central; but then $P$ is a direct factor of $G$, which implies $P=1$ by Lemma~\ref{lem: center}.
\end{proof}

As a (very) partial converse, we also get:

\begin{prop} \label{lemma:central ext}
Let $H$ be a Gagola $2$-group and let $G$ be a group such that $G/\Center (G) \cong H$ and $|\Center (G)| = 2$.   Then $G$ has an irreducible character that vanishes on all but one noncentral conjugacy class.
\end{prop}

\begin{proof}
Let ${\rm Z}_2(G)$ denote the second center of $G$, that is, the subgroup of $G$ containing ${\rm Z}(G)$ and such that ${\rm Z}_2(G)/{\rm Z}(G)={\rm Z}(G/{\rm Z}(G))$. Since $H$ is a $2$-group with a Gagola character, we know that $|\Center (H)| = 2$, thus $|{\rm Z}_2 (G)/\Center (G)| = 2$.  This implies that $|{\rm Z}_2 (G)| = 4$ and $|{\rm Z}_2 (G) \setminus \Center (G)| = 2$. If $g \in {\rm Z}_2 (G) \setminus \Center (G)$, then $g$ is not central in $G$, hence $|g^G|> 1$ and this forces $g^G$ to be the whole set $Z_2 (G) \setminus Z(G)$. Now, let $\chi$ be the Gagola character of $H$, and denote by $\hat\chi$ the character of $G$ obtained from $\chi$ by inflation. Since $\chi$ vanishes on all the noncentral conjugacy classes of $H$, it follows that $\hat\chi$ vanishes on all the noncentral conjugacy classes of $G$ except $g^G$. This proves the claim.  
\end{proof}

It is not difficult to see that the only groups $G$ of order $16$ with $|\Center (G)| = 2$ are the dihedral group, the generalized quaternion group, and the semi-dihedral group. In all cases, $H=G/\Center (G)$ is isomorphic to the dihedral group of order $8$ and it has a Gagola character. Thus, when $|H| = 8$, the only group with a Gagola character that has a central extension whose center has order $2$ is the dihedral group of order $8$, which has three such central extensions.   

Similarly, checking groups of order $32$ or $128$ with {\texttt{MAGMA}}, we find that for both orders there is also only one Gagola group $H$ that has a central extension whose center has order~$2$ (in both cases, $H$ is not an extra-special $2$-group). Again using {\texttt{MAGMA}}, in the two cases we obtain respectively $6$ central extensions of order $64$ and $8$ central extensions of order $256$.  On the other hand, we have found $64$ Gagola groups of order $512$ each having at least two central extension whose center has order $2$, hence, at least $128$ groups of order $1024$ as in Proposition~\ref{lemma:central ext}. We note that for the orders $16$, $64$ and $256$, all of the central extensions $G$ as in Proposition~\ref{lemma:central ext} for each order have the same character degree multi-set. On the other hand, for order $1024$ we see that there are at least $29$ different multi-sets that occur as character degree sets of the relevant central extensions.

\section{Central extensions Gagola groups: Frobenius groups} \label{examples}

The other special case of a Gagola group $H$ we wish to consider is when $H$ is not a group of prime-power order, and the unique minimal normal subgroup $N$ of $H$ is a Sylow $p$-subgroup of $H$ for a suitable prime $p$; in view of Lemma~\ref{gagola 1}, $(H,N)$ is then a Camina pair whose Camina kernel $N$ has an order coprime to the index $|H:N|$, which implies that $H$ is a Frobenius group with Frobenius kernel $N$.  In fact, by Theorem 6.2 of \cite{Gagola}, we have $N = {\rm O}_p (H)$ if and only if $H$ is a doubly transitive Frobenius group. 
Theorem~\ref{detailedThm7} yields that, for a group $G$ satisfying the equivalent conditions of that theorem and such that $H=G/{\rm Z}(G)$ is as above, the center ${\rm Z}(G)$ is a $p$-group: we now show that, in this situation and with the extra assumption that $|{\rm Z}(G)|=p$, the structure of $G$ is very limited.

\begin{thm} \label{th: semi-direct prod}
Let $G$ be a group having a nonlinear irreducible character that vanishes on all but one noncentral conjugacy class of $G$. Assume that $G/\Center (G)$ is a doubly transitive Frobenius group and, denoting by $N/{\rm Z}(G)$ its Frobenius kernel, let $p$ be the prime divisor of $|N/{\rm Z}(G)|$. If $|\Center (G)| = p$, then $G$ is a semi-direct product of a group $X$ acting on $N$ and one of the following occurs:
\begin{enumerate}[{\rm (a)}]
\item  $p=2$, $N$ is isomorphic to $Q_8$, $X$ has order $3$, and $G$ is isomorphic to ${\rm SL}_2 (3)$.
\item  $p$ is odd, and $N$ is an extra-special $p$-group of exponent $p$ and order $p^{2a+1}$ for some positive integer $a$. Also, $X$ is a $p'$-group that centralizes $\Center (N)$ and is isomorphic to a subgroup of ${\rm Sp}_{2a} (p)$.  
\end{enumerate}   

\end{thm} 

\begin{proof}
Note that, $G/{\rm Z}(G)$ being a doubly transitive Frobenius group, $N/{\rm Z}(G)$ is a chief factor of $G$ (hence, an elementary abelian $p$-group). Also, $G/N$ is a $p'$-group, so $N$ is a Sylow $p$-subgroup of $G$.  Because $N$ is normal, we may use the Schur-Zassenhaus theorem to see that $G$ has a Hall $p$-complement $X$.  
Let $\lambda$ be a nonprincipal irreducible character of $\Center (G)$.  We see that $X$ acts on the irreducible constituents of $\lambda^N$ and $X$ acts on the irreducible characters of $N/\Center (G)$.  If $\lambda$ extends to $N$, then by Gallagher's theorem (see Corollary 6.17 of \cite{Isaacs}), the linear characters of $N/\Center (G)$ act transitively by 
multiplication on the (linear) irreducible constituents of $\lambda^N$ and we may apply Glauberman's lemma (see Lemma 13.8 of \cite{Isaacs}) to see that $\lambda$ has an $X$-invariant extension $\theta$ to $N$.  Since $|G:N|$ and $|N|$ are coprime, we may use Corollary 8.16 of \cite{Isaacs} to see that $\theta$ has an extension $\psi\in\irr G$.  Observe that $\psi$ is an extension of $\lambda$ and so $\psi$ is linear.  This is a contradiction to Theorem \ref{detailedThm7} (3) which states that $\psi$ should have zeros in $N$, 
but since $\psi$ is linear, it does not vanish anywhere.  Therefore, $\lambda$ does not extend to $N$ and, in particular, $N$ is not abelian.
	
Recalling that $N/\Center (G)$ is a chief factor for $G$, by Problem 6.12 of \cite{Isaacs} we see that every nonprincipal irreducible character of $\Center (G)$ either extends to $N$, induces irreducibly to $N$, or is fully ramified with respect to $N/\Center (G)$.  Since the irreducible characters of $\Center (G)$ are invariant in $N$, they do not induce irreducibly, and we showed in the previous paragraph that they do not extend.  Thus, they are all fully-ramified.  Applying Proposition 7.7 of \cite{Hu}, this implies that $N$ is an extra-special $p$-group.  By Theorem~\ref{detailedThm7}, we have that the elements in $N \setminus \Center (G)$ form a conjugacy class of $G$.  

Now, assume that $p = 2$.  Since $N$ is nonabelian, we know that $N$ does not have exponent $2$.  On the other hand, since $|\Center (G)| = 2$, we see that every element in $N \setminus \Center (G)$ has order $4$.  Since the only extra-special group with a unique element of order $2$ is the quaternion group of order $8$ (see Satz III.13.8 of \cite{hup}) and $G/\Center (G)$ is a doubly transitive Frobenius group, we have that $|G/N| = |X| = 3$.  This implies that $G$ is isomorphic to ${\rm SL}_2 (3)$, as desired (see Theorem 19.10 of \cite{passman}). 
	
Assume then $p\neq 2$.  We know that $|N| = p^{2a+1}$ for a positive integer $a$.  Suppose $N$ has exponent $p^2$.  Since $N$ is extra-special, 
$N'$ is central and both $N/\Center (N)$ and $\Center (N)$ are elementary abelian.  It is not difficult to see that the map $n \mapsto n^p$ is a surjective homomorphism from $N$ to $\Center (N)$.  Clearly, the kernel of this homomorphism has index $p$ and contains $\Center (N)$.   It follows that $N \setminus \Center (N)$ contains elements of both order $p$ and $p^2$.  In other words, it is not possible for all elements of $N \setminus \Center (G)$ to be conjugate in this case.  Thus, $N$ must have exponent $p$.  We know that $X$ centralizes every element in $\Center (N)=\Center (G)$, and hence it is isomorphic to a subgroup of ${\rm Sp}_{2a} (p)$ (see \cite{winter}). 
\end{proof}

Using {\texttt{MAGMA}}, it is possible to construct groups as in the previous theorem in which the pairs $(|N|, X)$ are the following: $(3^3, Q_8)$, $(5^3, {\rm SL}_2 (3))$, and $(11^3,{\rm SL}_2 (5))$. We will prove that these are the only groups of this kind with $p$ odd and $a = 1$. In the following statement, we denote by $E_p$ an extra-special $p$-group of order $p^3$ and exponent~$p$.

\begin{thm}\label{doubly-transtive p^2}
Let $G$ be a group having a nonlinear irreducible character that vanishes on all but one noncentral conjugacy class of $G$. Assume that $G/\Center (G)$ is a doubly transitive Frobenius group and, denoting by $N/\Center(G)$ its Frobenius kernel, let $|N/\Center (G)|=p^2$ for a prime $p$. If $|\Center(G)|=p$, then $G$ is one of the following:
\begin{enumerate}[{\rm (a)}]
\item ${\rm SL}_2 (3)$.
\item $E_3 \rtimes Q_8$.
\item $E_5 \rtimes {\rm SL}_2 (3)$.
\item $E_{11} \rtimes {\rm SL}_2 (5)$. 
\end{enumerate}	
\end{thm}

\begin{proof}
From Theorem \ref{th: semi-direct prod}, $G$ is either ${\rm SL}_2 (3)$ or $G$ is a semi-direct product of an extra-special $p$-group $N$ of order $p^{2a+1}$ and exponent $p$ acted on by a group $X$ that centralizes $\Center (N)$ and is isomorphic to a subgroup of ${\rm Sp}_{2a} (p)$.   By assumption $a = 1$.  Recall that ${\rm Sp}_2 (p) \cong {\rm SL}_2 (p)$.  Also, we use Dickson's classification of the subgroups of ${\rm PSL}_2 (p)$ which can be found in Hauptsatz II.8.27 of \cite{hup}.  In particular, the classification of these groups includes cyclic groups of order dividing $(p-1)/2$, $p$, or $(p+1)/2$, the dihedral groups of orders dividing $p-1$ or $p+1$, the Frobenius groups whose Frobenius kernels have order $p$ and whose Frobenius complements have orders dividing $(p-1)/2$, groups isomorphic to $S_3$, $A_4$, $S_4$, and $A_5$.  Since $X$ will be the preimage of one of these groups and $p$ does not divide its order, we see that $|X| \le 120$.  This implies that $p \le 11$.   For the primes less than or equal to $11$, we can construct the actual semi-direct products using {\texttt{MAGMA}}. 
\end{proof}

Note that the groups in Theorem \ref{doubly-transtive p^2} (a), (b), and (c) are solvable, but in (d) we obtain a nonsolvable example.

It might be interesting to determine whether the groups in Theorem~\ref{doubly-transtive p^2} are the only groups satisfying the assumptions of Theorem~\ref{th: semi-direct prod} (b) (i.e., whether the positive integer $a$ in Theorem~\ref{th: semi-direct prod} (b) is necessarily $1$). In fact, we believe it is not possible to have groups of this kind where $a>1$; this would be equivalent to showing that, for $a>1$, ${\rm Sp}_{2a} (p)$ has no subgroups of order $p^{2a} - 1$ that are isomorphic to a Frobenius complement.  Unfortunately, we have not been able to prove it at this time.

Finally, let $G$ be a group such that $H=G/\Center(G)$ is a Gagola group, let $N$ be the unique minimal normal subgroup of $H$, and denote by $p$ be the prime divisor of $|N|$. The ``intermediate case" between the situations considered in Section 5 and in this section is when $N< {\rm O}_p(H)<H$; apart from the fact that $\Center(G)$ is a $p$-group by Theorem~\ref{detailedThm7}, at this point we have not found any other general restrictions on what can occur in this case, and we leave this for future investigation.

\end{document}